\documentclass[11pt,letterpaper]{amsart}
\usepackage{latexsym,amscd,amssymb, graphicx, color, amsthm, bm, amsmath, cancel, enumitem, ytableau}  
\usepackage{tikz}
\usepackage[margin=1in]{geometry}

\usepackage[all,cmtip]{xy}
\usepackage{mathtools}
\usepackage{rotating}

\usepackage[latin1]{inputenc}
\usepackage{amsfonts}

\usepackage{amsthm}
\usepackage{graphicx}
\usepackage{fullpage}

\usepackage{import}

\usepackage[colorlinks,linkcolor=blue,hyperindex,citecolor=orange]{hyperref}
\usepackage[backend=biber, style=numeric, maxbibnames=15]{biblatex}
\renewbibmacro{in:}{}

\usepackage{tikz-cd}

\usepackage{amssymb}
\usepackage{mathrsfs}
\usepackage{mathtools}
\usepackage{xcolor}
\usepackage{stmaryrd}
\usepackage{caption}
\usepackage{verbatim}
\usepackage{tcolorbox}
\usepackage{tikz}
\usetikzlibrary{arrows}
\usepackage{array}
\usepackage{float}
\usepackage{bbm}
\usetikzlibrary{shapes.multipart,positioning,decorations.pathreplacing,calc,intersections,folding}
\usepackage{rotating} 
\usepackage{graphicx} 

\usetikzlibrary{shapes.multipart,positioning,decorations.pathreplacing,calc,intersections,folding,backgrounds}
\tikzstyle{mydashed}=[dash pattern=on 1.5pt off 1.5pt]
\usetikzlibrary{lindenmayersystems}
\pgfdeclarelindenmayersystem{cayley}{
  \rule{F -> F [ R [F] [+F] [-F] ]}
  \symbol{R}{
    \pgflsystemstep=0.5\pgflsystemstep
  }
}
 \usepackage{bm}

\def\multiset#1#2{\ensuremath{\left(\kern-.3em\left(\genfrac{}{}{0pt}{}{#1}{#2}\right)\kern-.3em\right)}}

\newcommand{\SP}{\mathrm{P}}
\newcommand{\KSP}{\Bbbk\mathrm{P}}

\newcommand{\KE}{\Bbbk\mathrm{E}}

\newcommand{\MV}{\mathrm{M}}
\newcommand{\CV}{\mathrm{C}}
\newcommand{\SB}{\mathrm{H}}

\newcommand{\SH}{\mathrm{H}}
\newcommand{\VH}{\mathrm{H}}

\newcommand{\KSH}{\mathrm{H}}

\newcommand{\po}{\mathrm{PO}}

\newcommand{\co}{\mathrm{\Sigma}}
\newcommand{\kco}{\Bbbk\mathrm{\Sigma}}
\newcommand{\kcod}{\Bbbk\mathrm{\Sigma}^*}

\newcommand{\lc}{\overline{\mathrm{C}}}
\newcommand{\lcg}{\overline{\mathrm{cG}}}

\newcommand{\klcgd}{\Bbbk\overline{\mathrm{cG}}^*}

\newcommand{\Ex}{\mathrm{Ex}}
\newcommand{\ex}{\mathrm{ex}}

\newcommand{\BM}{\mathtt{M}}
\newcommand{\BP}{\mathtt{P}}
\newcommand{\BF}{\mathtt{F}}

\newcommand{\s}{\mathrm{s}}

\newcommand{\Min}{\mathrm{Min}}
\newcommand{\Max}{\mathrm{Max}}

\newcommand{\interior}{\mathrm{int}}
\newcommand{\closure}{\mathrm{cl}}
\newcommand{\starr}{\mathrm{star}}

\newcommand{\Des}{\mathrm{Des}}

\newcommand{\type}{\mathrm{type}}
\newcommand{\rank}{\mathrm{rk}}

\newcommand{\qsym}{QSym}
\newcommand{\hcg}{\mathrm{H_{cg}}}

\newcommand{\cc}{\mathfrak{c}}

\newcommand{\ha}{\mathcal{A}}
\newcommand{\Descent}{\mathrm{Des}}

\newcommand{\Peak}{\mathrm{Peak}}

\newcommand{\supp}{\mathrm{supp}}

\newtheorem{theorem}{Theorem}[section]
\newtheorem{proposition}[theorem]{Proposition}
\newtheorem{lemma}[theorem]{Lemma}
\newtheorem{corollary}[theorem]{Corollary}

\newtheorem{question}[theorem]{Question}
\newtheorem{example}[theorem]{Example}
\newtheorem{definition}[theorem]{Definition}

\newtheorem{remark}[theorem]{Remark}

\newtheorem*{theorem_polynomial_invariants_enumeration}{Theorem~\ref{poly_cg_eta}, \ref{poly_cg_varphi'}}
\newtheorem*{corollary_reciprocity}{Corollary~\ref{reciprocity_eta_zeta}, \ref{reciprocity_cg_varphi'}}
\newtheorem*{theorempreview1}{Theorem~\ref{supersolvable_equivalent_condition}}
\newtheorem*{theorempreview2}{Theorem~\ref{ab_character_supersolvable}}
\newtheorem*{theorempreview3}{Theorem~\ref{cd_varphi'_peaks}, \ref{cd_varphi_peaks}}

\title{Convex Geometries via Hopf Monoids: Combinatorial Invariants, Reciprocity, and Supersolvability}
\author{Yichen Ma}
\address{Department of Mathematics \\ Cornell University \\ Ithaca, NY, 14853 \\ USA}
\email{ym476@cornell.edu}

\begin{document}

\maketitle{}
\begin{abstract}
We study the Hopf monoid of convex geometries, which contains partial orders as a Hopf submonoid, and investigate the combinatorial invariants arising from canonical characters. Each invariant consists of a pair: a polynomial and a more general quasisymmetric function. We give combinatorial descriptions of the polynomial invariants and prove combinatorial reciprocity theorems for the Edelman-Jamison and Billera-Hsiao-Provan polynomials, which generalize the order and enriched order polynomials, respectively, within a unified framework. For the quasisymmetric invariants, we show that their coefficients enumerate faces of certain simplicial complexes, including subcomplexes of the Coxeter complex and a simplicial sphere structure introduced by Billera, Hsiao, and Provan. We also examine the associated $ab$- and $cd$-indices. We establish an equivalent condition for convex geometries to be supersolvable and use this result to give a geometric interpretation of the $ab$- and $cd$-index coefficients for this class of convex geometries.
\end{abstract}

\section{Introduction}\label{sec:Introduction}

Joyal \cite{MR0633783} introduced \emph{species} as a unified framework for studying collections of combinatorial objects, which was further developed by Bergeron, Labelle, and Leroux \cite{MR1629341}. The study of Hopf algebras and related combinatorial structures dates back to the work of Joni and Rota \cite{joni1979coalgebras}, with the underlying ideas of merging and breaking combinatorial objects. The theory of \emph{Hopf monoids} in species, developed by Aguiar and Mahajan \cite{MR2724388}, takes a categorical approach. Unlike a Hopf algebra, which is generated by unlabeled objects, a Hopf monoid is generated by labeled objects, thereby encoding more information. This theory has been fruitfully applied to study various structures such as partial orders, graphs, simplicial complexes, and generalized permutahedra \cite{MR4651496, ABV, ArdilaSanchez2022, Castillo2024Hopf}. 

A central tool in this theory is the notion of \emph{characters}, which are multiplicative functions on Hopf algebras and admit natural Hopf monoid analogues. It was studied by Aguiar, Bergeron, and Sottile \cite{MR2196760}, and later explored to Hopf monoids by Aguiar and Bastidas \cite{ABV}.  Each character yields a pair of combinatorial invariants: a polynomial and a more general quasisymmetric function via the universal map as studied in \cite{MR2196760}. 


One particularly rich family of combinatorial structures that can be studied via Hopf monoids is that of \emph{convex geometries}, introduced by Edelman \cite{Edelman} as a generalization of order ideals in partially ordered sets. Edelman and Jamison \cite{EJ} characterized convex geometries via closure operators and provided examples including those arising from Euclidean closure in $\mathbb{R}^n$, partial orders, oriented matroids, graphs, and hypergraphs. Extending the work of \cite{ABV}, we generalize the study of characters from partial orders to convex geometries. In particular, we show that the \emph{polynomial invariants} associated with the canonical characters, $\eta$, $\zeta$, $\varphi$, and $\varphi'$, correspond to the counting functions which generalize the order polynomial \cite{MR0269545}, the strict order polynomial, and the enriched order polynomial \cite{Stembridge}. For $\eta$, $\zeta$, and $\varphi'$, we obtain the polynomials first introduced by Edelman-Jamison \cite{EJ} and Billera-Hsiao-Provan \cite{BHP}.

\begin{theorem_polynomial_invariants_enumeration}
Let $\chi^\psi$ denote the polynomial invariant associated with the character $\psi$. Let $g$ be a convex geometry on ground set $I$.
    \begin{itemize}
        \item[(1)]$\chi^\eta_I(g)(n)$ counts the number of extremal functions $f: I\rightarrow [n]$. 
        \item[(2)]$\chi^\zeta_I(g)(n)$ counts the number of strictly extremal functions $f: I \rightarrow [n]$.
        \item[(3)] $\chi^{\varphi'}_I(g)(n)$ counts the number of enriched extremal functions $f: I \rightarrow \llbracket n\rrbracket$.
            \end{itemize}

In particular, $\llbracket n \rrbracket = \{\overline{1}<1<...<\overline{n}<n\}.$
\end{theorem_polynomial_invariants_enumeration}

With this understanding, we show a collection of combinatorial reciprocity results from the unified  perspective of characters and their associated polynomial invariants. These include the Edelman-Jamison reciprocity \cite[Theorem 4.7]{EJ} and the Billera-Hsiao-Provan reciprocity \cite[pp.16]{BHP}.

\begin{corollary_reciprocity}
Let $g$ be a convex geometry on ground set $I$.
    \begin{itemize}
        \item[(1)]  $(-1)^{|I|}\chi^\eta_I(g)(-n) = \chi^\zeta_I(g)(n)$.
        \item[(2)] $\chi^{\varphi'}_I(g)(-n) = (-1)^{|{I}|}\chi^{\varphi'}_I(g)(n).$
        \item[(3)] $\chi^{\varphi}_I(g)(-n) = (-1)^{|I|}\chi^{\varphi}_I(g)(n).$
    \end{itemize}
\end{corollary_reciprocity}

The notion of \emph{quasisymmetric invariants} (flag $f$-vectors) of Hopf algebras arises from the coefficients of the quasisymmetric function canonically associated with a character $\psi$, as discussed in \cite{MR2196760}. This concept has a natural analogue in Hopf monoids. For convex geometries, we give a geometric interpretation of the flag $f$-vectors associated with canonical characters, describing them as enumerations over certain subcomplexes of the Coxeter complex (using the hyperplane arrangement language as in \cite{MS}) and over specific faces in the simplicial sphere structure introduced by Billera, Hsiao, and Provan \cite{BHP}.

In some contexts, convex geometries are also referred to as \emph{antimatroids}, in which the feasible sets are precisely the complements of convex sets. Antimatroids form an important subclass of \emph{greedoids}. Armstrong \cite{Armstrong2007TheSO} developed the theory of \emph{supersolvable antimatroids} in analogy with supersolvable lattices, originally introduced by Stanley \cite{Stanley1972SupersolvableL}. Supersolvable antimatroids naturally appear in the study of closure operators \cite{MR1225869}, Coxeter groups \cite{Armstrong2007TheSO}, and matroids \cite{MR4245154}. Recently, Backman and Danner \cite{backman2024convex} proved that the building sets on a finite meet-semilattice form a supersolvable convex geometry, and applied this result to unify and extend a collection of results in algebraic geometry, matroid theory, and related areas. In Section \ref{supersolvable_convex_section}, we establish a geometric characterization of supersolvable convex geometries on finite ground sets.

\begin{theorempreview1}
Let $g$ be a convex geometry. Let $p_1$,..., $p_k$ be partial orders such that $V_{p_1}$,..., $V_{p_k}$ are the maximal convex cones in $V_g$, the order complex associated to $g$.
Then the following statements are equivalent.\begin{enumerate}
\item[(1)] $\bigcap_{i=1}^k V_{p_i}$ contains at least one chamber. That is, there exists a partial order $p_0$ on $I$ with $V_{p_0} = \bigcap_{i=1}^k V_{p_i}$. 
    \item[(2)] $g$ is supersolvable.
\end{enumerate}
\end{theorempreview1}

Using this geometric description, we express the $ab$-indices (from the flag $h$-vector) of the canonical characters $\eta$ and $\zeta$ on supersolvable convex geometries in terms of \emph{descents}. We use $\Psi^\psi_g$ denote the $ab$-index associated to a given convex geometry $g$ and a character $\psi$. For two linear orders $\ell_1$ and $\ell_2$ on $I$, let $\left( \begin{array}{c}
     \ell_1  \\
     \ell_2
\end{array} \right )$ denote the two-line permutation determined by $\ell_1$ and $\ell_2$.

\begin{theorempreview2}
Let $g$ be a supersolvable convex geometry. For $S\subseteq [n-1]$, let $m(a,b)_S$ denote the degree $n-1$ $ab$-monomial with $b$ on position $s$ for each $s\in S$. Fix any linear order $l_0$, $l_0'$ satisfying $\overline{l_0} \in V_{{p_0}}$, $l'_0 \in V_{p_0}$. 
\begin{enumerate}
\item $[m(a,b)_S]\Psi^{\eta}_g = \# \left \{\ell \in V_g \mid \Descent (\left ( \begin{array}{c}
     \ell'_0  \\
     \ell 
\end{array} \right ) ) = S\right \}.$
    \item $[m(a,b)_S]\Psi^{\zeta}_g = \#\left \{\ell \in V_g \mid \Descent (\left ( \begin{array}{c}
     \ell_0  \\
     \ell 
\end{array} \right ) ) = S\right \}.$

\end{enumerate}
\end{theorempreview2}

The notion of the $cd$-index originated in the combinatorial study of convex polytopes and was introduced by Bayer and Klapper \cite{Bayer1991ANI}. Bayer and Billera \cite{BayerBillera1985} generalized the Dehn-Sommerville relations to arbitrary polytopes and Eulerian posets, and these generalized relations were shown in \cite{Bayer1991ANI} to be equivalent for a graded poset to admit a $cd$-index with integer coefficients. More recently, Aguiar, Bergeron, and Sottile \cite{MR2196760} showed that for a specific class of characters, namely \emph{odd} characters, the associated flag $f$-vectors satisfy the Bayer-Billera relations. As a result, the notion of $cd$-indices for odd characters in Hopf algebras is well-defined, along with an analogue for odd characters in Hopf monoids.

Applying our geometric characterization of supersolvable convex geometries, we compute the $cd$-indices of the canonical odd characters $\varphi'$ and $\varphi$ on this class and provide geometric descriptions using the concept of \emph{peak}. We use $\Phi^\psi_g$ denote the $cd$-index associated to a given convex geometry $g$ and a character $\psi$. Note the $cd$-monomial $c$ has degree 1 and $d$ has degree 2.

\begin{theorempreview3}
Let $g$ be a supersolvable convex geometry. Let $m(c,d)_S$ denote the degree $n-1$ $cd$-monomial such that the degrees of the initial segments ending in $d$'s are precisely the elements in $S \subseteq [n-1]$. Fix any linear orders $l_0$ and $l_0'$ satisfying $\overline{l_0} \in V_{p_0}$ and $l_0' \in V_{p_0}$.
\begin{enumerate}

    \item 
$[m(c,d)_S]\Phi^{\varphi'}_g = 2^{|S|+1} \cdot \#\left \{\ell \in V_g \mid \Peak(\left ( \begin{array}{c}
     \ell_0  \\
     \ell 
\end{array} \right )) = S \right \}.$
\item 
$[m(c,d)_S]\Phi^{\varphi}_{g} = 2^{|S|+1} \cdot \# \left\{\ell \in V_g \mid \Peak(\left ( \begin{array}{c}
     \ell'_0  \\
     \ell 
\end{array} \right )) = S\right \}.$
\end{enumerate}

\end{theorempreview3}

In this manuscript, we work with a finite ground set $I$ and a field $\Bbbk$ with $\mathrm{char}(\Bbbk) \neq 2$. The manuscript is organized as follows. In Section \ref{backgrounds} and \ref{background_geo}, we provide a brief introduction to the main topics. In Section \ref{convex_geometries}, we compute four canonical characters of the Hopf monoid of convex geometries, derive the reciprocity results of the associated polynomial invariants, and describe the corresponding quasisymmetric invariants. In Section \ref{ab_cd_section}, we present equivalent conditions for a convex geometry to be supersolvable and compute the $ab$ and $cd$-indices associated with the canonical characters on supersolvable convex geometries. In Section \ref{supersolvable_closure_operators}, we extend some of the preceding discussions to supersolvable closure operators.

\section{Preliminaries on Hopf monoids}\label{backgrounds}

\subsection{Hopf monoids}
We fix a field $\Bbbk$. A \emph{combinatorial Hopf algebra} is a pair $(\VH,\psi)$, where $\VH = \bigoplus_{n \geq 0}\VH _n$ is a graded connected Hopf algebra over $\Bbbk$, and $\psi: \VH \rightarrow \Bbbk$ a \emph{character}. That is, $\psi$ is $\Bbbk$-linear and multiplicative. For a detailed introduction of Hopf algebra, we refer to \cite{grinberg2020hopf, Martin2012LectureNO}.

Unlike Hopf algebras, which discuss unlabeled objects, investigations on \emph{Hopf monoids} work on labeled objects. For a comprehensive reference of Hopf monoids, we refer to \cite{MR4651496, MR2724388}. In this section, we give a brief introduction. 

A \emph{set species} is a functor $\mathrm{set^\times} \rightarrow \mathrm{Set}$, where $\mathrm{set^\times}$ is the category of finite sets and bijections, and $\mathrm{Set}$ is the category of all sets and maps. A \emph{vector species} is a functor $\mathrm{set^\times} \longrightarrow \mathrm{Vec}$, where $\mathrm{Vec}$ is the category of vector spaces and linear maps. We may obtain a \emph{linearization} of a set species $\SP$ by defining $\KSP[I]$ to be the vector space spanned by $\SP[I]$, and $\KSP[\sigma]$ to be be the unique linear map extending $\SP[\sigma]$ for each bijection $\sigma: I\rightarrow J$.

A \emph{monoid vector species} is a vector species $\MV$ endowed with morphisms of species $\mu: \MV \cdot \MV \rightarrow \MV$ and $\iota: \Bbbk1 \rightarrow \MV$ satisfying the following axioms.
\begin{enumerate}
    \item $(\emph{Naturality})$ For finite sets $I$, $J$, each decomposition $I = S\sqcup T$, and each bijection $\sigma: I\rightarrow J$, the following diagram commutes.$$\begin{tikzcd}
{\MV[S]\otimes \MV[T]} \arrow[d, "{\MV[\sigma|_S]\otimes \MV[\sigma|_T]}"'] \arrow[r, "{\mu_{S,T}}"] & {\MV[I]} \arrow[d, "{\MV[\sigma]}"] \\
{\MV[\sigma(S)]\otimes \MV[\sigma(T)]} \arrow[r, "{\mu_{\sigma(S),\sigma(T)}}"']                       & {\MV[J]}                            
\end{tikzcd} $$
    
    \item $(\emph{Associativity})$ For each decomposition $I = R\sqcup S\sqcup T$, the following diagram commutes. $$\begin{tikzcd}
 \MV[R] \otimes \MV[S] \otimes  \MV[T] \arrow[r, "\mathrm{id} \otimes \mu_{S,T}"] \arrow[d, "\mu_{R,S} \otimes id"'] & \MV[R] \otimes \MV[S\sqcup T] \arrow[d, "\mu_{R, S\sqcup T}"] \\
 \MV[R\sqcup S] \otimes \MV[T] \arrow[r, "\mu_{R\sqcup S,T}"]                                                 & \MV[I]                          
 \end{tikzcd} $$ 
    \item $(\emph{Unitality})$ For each finite set $I$, the following diagrams commute. $$
\begin{tikzcd}
\MV[I] & \MV[\emptyset] \otimes \MV[I] \arrow[l, "\mu"']                                      &  & \MV[I] \otimes \MV[\emptyset] \arrow[r, "\mu"]                                       & \MV[I] \\
     & \Bbbk \otimes \MV[I] \arrow[u, "\iota \otimes \mathrm{id}_I"'] \arrow[lu, "\cong"] &  & \MV[I] \otimes \Bbbk \arrow[u, "\mathrm{id}_T \otimes \iota"] \arrow[ru, "\cong"'] &     
\end{tikzcd}$$
    
\end{enumerate}

A \emph{comonoid vector species} is a vector species $\CV$ endowed with morphisms of species $\Delta: \CV \rightarrow \CV\cdot \CV$, $\eta: \CV\rightarrow \Bbbk1$ satisfying naturality, coassociativity, and counitality. That is, $\Delta$ and $\eta$ make the duals of the above diagrams commute.

A species $\SH$ is a \emph{bimonoid vector species} if it is both a monoid vector species and a comonoid vector species and satisfies the additional \emph{compatibility} axiom. That is, for each finite set $I$, decompositions $I= S_1\sqcup S_2 = T_1 \sqcup T_2$, consider the pairwise intersection as follows, $$A:= S_1\cap T_1, \text{ }B:= S_1\cap T_2, \text{ }C:=S_2\cap T_1, \text{ }D:= S_2\cap T_2. $$ Let $\beta$ denote the braiding map for monoid species. Then the compatibility axiom can be illustrated by the following commutative diagrams.
$$\begin{tikzcd}
{\VH[A]\otimes \VH[B]\otimes \VH[C]\otimes \VH[D]} \arrow[rr, "\mathrm{id}\otimes \beta\otimes\mathrm{id}"] &                                               & {\VH[A]\otimes\VH[C]\otimes\VH[B]\otimes\VH[D]} \arrow[d, "{\mu_{A,C}\otimes\mu_{B,D}}"] \\
{\VH[S_1]\otimes \VH[S_2]} \arrow[r, "{\mu_{S_1,S_2}}"'] \arrow[u, "{\Delta_{A,B}\otimes\Delta_{C,D}}"] & {\VH[I]} \arrow[r, "{\Delta_{T_1,T_2}}"'] & {\VH[T_1]\otimes \VH[T_2]}          \end{tikzcd} $$ $$\begin{tikzcd}
{\VH[\emptyset]\otimes\VH[\emptyset]} \arrow[r, "\eta_\emptyset\otimes\eta_\emptyset"] \arrow[d, "{\mu_{\emptyset,\emptyset}}"'] & \Bbbk\otimes\Bbbk \arrow[d, no head, Rightarrow, no head] & \Bbbk \arrow[d, no head, Rightarrow, no head] \arrow[r, "\iota_{\emptyset}"]           & {\VH[\emptyset]} \arrow[d, "{\Delta_{\emptyset,\emptyset}}"] \\
{\VH[\emptyset]} \arrow[r, "\eta_{\emptyset}"]                                                                                            & \Bbbk                                             & \Bbbk\otimes\Bbbk \arrow[r, "\iota_{\emptyset}\otimes \iota_{\emptyset}"] & {\VH[\emptyset]\otimes \VH[\emptyset]}                          
\end{tikzcd} $$ $$\begin{tikzcd}
                                                                                 & {\VH[\emptyset]} \arrow[rd, "\eta_{\emptyset}"] &         \\
\Bbbk \arrow[ru, "\iota_\emptyset"] \arrow[rr,  no head, Rightarrow, no head] &                                                      & \Bbbk
\end{tikzcd} $$

A \emph{morphism} between two (co)monoid species is a morphism of species satisfying unitality and (co)multiplicity. Let $\VH$ be a bimonoid vector species. Let $\mathrm{End}(\VH)$ denote the set of all morphisms of vector species $f: \VH\rightarrow \VH$. The \emph{convolution product} $f*g$ of
two morphisms $f$ and $g$ of vector species is defined by
$$(f*g)_I (x) =\sum_{S\sqcup T = I}f_S(x|_S)\cdot g_T(x/_S)$$ for all finite sets $I$ and all $x\in \SB[I]$. we say that $\VH$ is a \emph{Hopf monoid species} if the identity map $\mathrm{id}$: $\VH \rightarrow \VH$ is invertible in $\mathrm{End}(\VH)$ with respect to the convolution product. The invertible map is called the \emph{antipode} of $\VH$ and denoted as $\s$.

\subsection{Characters, polynomial invariants and reciprocity}
For detailed discussions on the definitions of characters, polynomial invariants, and dual bases on the Hopf monoid of partial orders and labeled partial orders, we refer to \cite{ABV}. Here, we review the basic concepts.

Let $\MV$ be a monoid species. A \emph{character} on $\MV$ is a morphism of monoid species $\psi: \MV\rightarrow \KE$, where $\KE$ is the species sending every finite set to $\Bbbk$. Specifically, $\psi$ consists of a collection of linear maps $$\psi_I: \MV[I]\rightarrow \Bbbk,$$ one for each finite set $I$, subject to naturality, multiplicativity and unitality. One important fact is that given a connected Hopf monoid, the set of characters forms a group under the convolution product \cite[Definition 1.1, Hopf algebra analogue]{MR2196760}. The inverse of a character $\psi$ is determined by $\psi \circ \s$.

The \emph{polynomial invariant} associated with $\psi$ is as the following summation. $$\chi_I(x)(n) = \sum_{I = S_1\sqcup ... \sqcup S_n}(\zeta_{S_1} \otimes...\otimes  \zeta_{S_n}) \circ \Delta_{S_1,...,S_n}(x)$$ for $x\in \KSH[I]$ and $n\in\mathbb{N}$. The sum runs over all weak compositions $(S_1,...,S_n)$ of $I$. \cite[Proposition 3.1]{ABV} showed that $\chi$ is indeed a polynomial.

\begin{proposition}\label{general_reciprocity}\textnormal{\cite[Proposition 3.7]{ABV}}
Let $\VH$ be a Hopf monoid, $\psi$ a character on $\VH$ and $\chi^\psi$ its associated polynomial invariant. Then
\begin{equation}
    \chi^\psi_I(x)(-1) = \psi_I(\s_I(x)).
\end{equation}
More generally, for every $n \in \mathbb{N}$, we have \begin{equation}
    \chi^\psi_I(x)(-n) = \chi^\psi_I(\s_I(x))(n).
\end{equation} 
\end{proposition}

The above proposition is a \emph{reciprocity} result of a very general nature. This explains why an explicit antipode formula is important: such information allows us to compute the values of all polynomial invariants at negative integers. While the polynomial invariant depends on the specific character, the antipode only depends on the Hopf monoid structure. The antipode acts as a universal link between the values of the invariants at positive and negative integers.

While we have explicit antipode formulas, such as the Hopf monoid versions of \emph{Takeuchi's formula} \cite[Proposition 8.13]{MR2724388} and \emph{Milnor-Moore's formula} \cite[Proposition 8.14]{MR2724388}, determining a cancellation-free antipode formula remains an open problem for many Hopf monoids. For instance, a cancellation-free formula for the antipode of convex geometries is still unknown. We will discuss the Hopf monoid of convex geometries in Section \ref{convex_geometries}. Consequently, the general reciprocity may appear uninteresting but involves extensive technical computations. However, if $\psi$ is an \emph{odd character}, meaning $\overline{\psi}_I = (-1)^{|I|}\psi_I$ is the inverse of $\psi$ in the character group, we have \emph{self-reciprocity} for the associated polynomial invariant.

\begin{proposition}\label{self_reciprocity}\textnormal{\cite[Proposition 3.13]{ABV}}
Let $\chi^\psi$ be the polynomial invariant associated with an odd character $\psi$, then $$ \chi^\psi_I(x)(-n) = (-1)^{|I|}\chi^\psi_I(x)(n)$$ for all $x\in\KSH[I]$. 
\end{proposition}

This means that, when $\psi$ is odd, the value of the polynomial invariant $\chi^\psi$ at negative integers depends only on the size of the ground set $I$.

\subsection{Quasisymmetric Invariants}

Let $\co$ be the set species of compositions. That is, for each finite $I$. $\co[I]$ is the set of compositions $F\vDash I$. Let $\kco$ be its linearization. We denote $\BP$ as the basis corresponding to each set composition in $\kco$. For example, if $F\in \co[I]$, then $\BP_F$ is the basis element corresponding to $F$ in $\kco[I]$. Let $\kcod$ be the dual species of $\kco$. That is, $\kcod[I]$ is the dual space of $\kco[I]$. Let $\BM$ be the dual basis of $\BP$.

in \cite[section 12.4]{MR2724388}, Aguiar and Mahajan discussed the Hopf monoid structures of $\kco$ and $\kcod$. Specifically, $\kcod$ is a Hopf monoid vector species with product and coproduct on the basis $\{\mathtt{M}\}$ as follows.

For $\mathtt{M}_F, \mathtt{M}_G$, $\mathtt{M}_F\cdot \mathtt{M}_G$ is the sum of all quasi-shuffles.

For $I = S\sqcup T$, $$\triangle_{S,T}(\mathtt{M}_F) = \begin{cases} \mathtt{M}_{F_1}\otimes \mathtt{M}_{F_2} & \text{if $F = F_1 \cdot F_2$, $S$ consists of minimal blocks in $F$,} \\ 0 & \text{if $S$ is not an initial segment in $F$.} \end{cases}$$

Let $\zeta: \kcod \rightarrow \KE$ be as follows. For each finite set $I$, $$\zeta_I(\mathtt{M}_F) = \begin{cases} 1 & \text{if $F = (I) $ or $I = \emptyset $}, \\ 0 & \text{otherwise.}  \end{cases} $$

\begin{theorem}\textnormal{\cite[Theorem 11.19]{MR2724388}}\label{universal_Hopf_monoid}
For any Hopf monoid vector species $\VH$ with a character $\psi: \VH \rightarrow \KE$, there exists a unique morphism of Hopf monoids $f: \VH \rightarrow \kcod$ such that the following diagram commutes.

$$\begin{tikzcd}
\VH \arrow[rr, "f^\psi", dashed] \arrow[rd, "\psi"'] &     & \kcod \arrow[ld, "\zeta"] \\
                                                 & \KE &                          
\end{tikzcd}$$

In the above diagram, $f^\psi$ is defined as follows. For $x\in \VH[I]$, \begin{equation}\label{quasi_invariant}
  f^\psi_I(x) = \sum_{F \vDash I}\psi_F \Delta_F(x) \mathtt{M}_F.  
\end{equation}
\end{theorem}

Fix $F = (F_1,...,F_k)\vDash I$, let $f_F^\psi(x) = \psi_F\Delta_F(x)$. Let $\alpha = (\alpha_1,...,\alpha_k)$ be an \emph{integer composition} of $|I|$. That is, $\alpha_1+...+\alpha_k = |I|$ and $\alpha_i \neq 0$. The associated \emph{quasisymmetric invariant}, or \emph{flag $f$-vector} with $\psi$ is the vector $(f^\psi_\alpha(x))_\alpha$ indexed by $\alpha \vDash |I|$ with
$$f_\alpha^\psi (x) = \sum_{F:\, \type(F)=\alpha} f_F^\psi (x). $$

The Hopf algebra analogy of Theorem \ref{universal_Hopf_monoid} is as follows.  Consider the combinatorial Hopf algebra $(\qsym,\eta_Q)$.  Let $\BM_\alpha$ denote the \emph{monomial quasi-symmetric function} indexed by $\alpha = (\alpha_1,...,\alpha_k)$, which forms a linear basis of $\qsym$. Specifically, $$
    \BM_\alpha := \sum_{i_1<i_2<...<i_k}x^{\alpha_1}_{i_1} x^{\alpha_2}_{i_2}...x^{\alpha_k}_{i_k}.$$
The character $\eta_Q: \qsym \rightarrow \Bbbk$ is defined as follows.
$$\eta_Q(\BM_\alpha) = \begin{cases}
    1 & \text{if }\alpha = (n) \text{ or }(), \\
    0 & \text{otherwise.}
\end{cases}$$

\begin{theorem}\textnormal{\cite[Theorem 4.1]{MR2196760}}\footnote{The authors denoted $\Upsilon$ as $\Psi$, and $\eta_Q$ as $\zeta_Q$.} For any combinatorial Hopf algebra $(\VH,\psi)$,  there exists a unique morphism of combinatorial Hopf algebras as follows.
$$
    \Upsilon:(\VH,\psi) \rightarrow (\qsym, \eta_Q).
$$
Specifically, for $x\in \VH$ \begin{equation}
    \Upsilon(x) = \sum_{\alpha \vDash n}\eta_\alpha(x)\BM_\alpha
\end{equation} where $\eta_\alpha$ is defines by the composition $$\VH \xrightarrow{\triangle^{(k-1)}}\VH^{\otimes k} \twoheadrightarrow \VH_{\alpha_1}\otimes ... \otimes \VH_{\alpha_k} \rightarrow \Bbbk.$$

where the map $\twoheadrightarrow$ is the tensor product of the canonical projections onto the
homogeneous components $\VH_{\alpha_i}$.
\end{theorem}

Hence, $f^\psi_\alpha(x)$ corresponds to the coefficient of $\BM_\alpha$ in $\Upsilon(x)$. We will elaborate further on the Hopf algebra arguments when we discuss the $cd$-indices of canonical odd characters in Section~\ref{cd_index}.

\subsection{Flag $h$-vectors and $ab$, $cd$-indices of Hopf monoids}

For a detailed discussion of the $cd$-index on Eulerian posets and polytopes, we refer to \cite{Bayer2019TheA}. Below, we provide an introduction to the $ab$ and $cd$-indices in the context of Hopf monoids.

Analogous to the definition of the flag $h$-vector for a polytope, the \emph{flag $h$-vector} for $\psi$ is determined by the following relation. $$h^\psi_\alpha(x) = \sum_{\beta \leq \alpha}(-1)^{l(\beta)-l(\alpha)}f^\psi_\alpha(x),$$ where $\beta \leq \alpha$ means $\beta$ can be obtained by merging blocks in $\alpha$.

For $\alpha = (\alpha_1,...,\alpha_{k}) \vDash n$, let $m(a,b)_\alpha$ denote the $ab$-monomial of degree $n-1$ with $b$'s on position $\alpha_1, \alpha_1+\alpha_2, ...,\alpha_1+...+\alpha_{k-1}.$ For example, if $\alpha = (2,1,1,5,3)\vDash 12$ then the associated $ab$-monomial is$m(a,b)_\alpha = abbbaaaabaa$.

Denote $\Psi^\psi_{x}(a,b)$ the \emph{$ab$-index} associated with $\psi$ on $h\in H[I]$. That is, $$\Psi^\psi_{x}(a,b) = \sum_{\alpha \vDash |I|}h^\psi_{\alpha}m(a,b)_{\alpha}.$$ Let $c= a+b$, $d= ab+ba$. The \emph{cd-index} \cite[Definition 2.6]{Bayer2019TheA}, if exists, is the polynomial $\Phi(c,d)$ such that $\Phi(c,d) = \Psi(a,b)$. 

\begin{theorem}[Section 5, \cite{MR2196760}] If the character $\psi$ is odd, then for any $x\in\VH[I]$ the vector $(f_\alpha^\psi (x))_\alpha$ satisfies the \emph{Bayer-Billera relations} \textnormal{\cite[Theorem 2.1]{BayerBillera1985}} (or \emph{generalized Dehn-Sommerville equations}).
\end{theorem}   

\begin{theorem}[Theorem 4, \cite{Bayer1991ANI}]
        The $cd$-index of $x$ exists if and only if the vector $(f^\psi_\alpha(x))_\alpha$ satisfies the Bayer-Billera relations.
\end{theorem}

Hence, we have the notion of the $ab$-index for every character and the $cd$-index for every odd character. We will elaborate on the descriptions of these indices in Section~\ref{ab_cd_section}.

\section{Preliminaries on hyperplane arrangements}\label{background_geo}

In this section, we provide background on hyperplane arrangements and discuss bijections between geometric and combinatorial descriptions of the braid arrangement on a finite set. We refer the reader to \cite{MS} for an exhaustive discussion of topics related to hyperplane arrangements.

\subsection{Hyperplane Arrangement}

Let $V$ be a vector space over $\mathbb{R}$. A \emph{hyperplane} on $V$ is a codimension-one affine subspace of $V$. A \emph{half-space} is a subset of $V$ which consists of elements on one side of a hyperplane. Hence given a hyperplane $H$, it has two associated half spaces, and their intersection is precisely $H$.  The bounding hyperplane is the \emph{boundary} of the half-space. By convention, a half-space is closed, namely it contains its boundary. The \emph{interior} of the half-space $\mathrm{H}$ is the half-space minus its boundary, and we denote it is $\interior(\mathrm{H})$. 

A \emph{hyperplane arrangement} $\ha$ is a finite set of hyperplanes in a finite-dimensional real vector space $V$. $V$ is called the \emph{ambient space} of $\ha$. The intersection of all hyperplanes in $\ha$ is called the \emph{center}. We say $\ha$ is \emph{essential} if its center is the origin. 

\subsection{Faces, Chambers, Stars}

A \emph{face} $F$ of $\ha$ is a subset of $V$ obtained by intersecting half-spaces of $\ha$, with at least one associated half-space chosen for each hyperplane. Note the center of $\ha$ is a face and we denote it as $O$.

The \emph{interior} of $F$ is the subset of $F$ obtained by intersecting $F$ with the interiors of those half-spaces used to define $F$ whose boundary does not contain $F$. 

Let $\Sigma[\ha]$ denote the set of faces in $\ha$. It is a graded poset under inclusion, with $O$ as its minimum element. Each face $F$ has a dimension, and the rank of $F$ is obtained by $$\rank(F) = \dim(F) - \dim (O).$$ 

The rank of the poset $\Sigma[\ha]$ equals the rank of $\ha$. A maximal face of $\Sigma[\ha]$ is called a \emph{chamber}. We denote the set of chambers in $\ha$ as $\Gamma[\ha]$. A rank-one face is called a \emph{vertex}, and a rank-two face is called an \emph{edge}.

The intersection of two faces is another face, so the meet exists in $\Sigma[\ha]$. We denote the meet of $F$ and $G$ as $F\wedge G$. Note, the join of two faces may not exist, and it exists precisely when $F$ and $G$ have a common upper bound. In this case, we denote the \emph{join} of $F$ and $G$ as $F\vee G$. In summary, $\Sigma[\ha]$ is a graded meet-semilattice. We say $F$ is a face of $G$ if $F \leq G$.

Every face $F$ has an \emph{opposite face}, denoted $\overline{F}$, which is given by $$\overline{F}  = \{-x \mid x\in F\}.$$ The \emph{opposition map} on faces $$\Sigma[\ha] \rightarrow \Sigma[\ha] \hspace{10mm} F\mapsto \overline{F} $$ is an order-preserving involution. That is, we have $$\overline{\overline{F}} = F \hspace{10mm} F\leq G \Rightarrow \overline{F}\leq \overline{G}.$$

Since chambers are faces, the opposition map can be restricted on the set of chambers $\Gamma[\ha]$. That is, every chamber $C$ has an opposite chamber $\overline{C}$.

For a face $F$, the \emph{star} of $F$ is the set of faces of $\ha$ which are greater than $F$. In particular, the star of a chamber is a singleton consisting of the chamber itself, while the star of the central face $O$ is the set of all faces. The \emph{top-star} of $F$ is the set of chambers which are greater than $F$.

Examples of both the star and the top-star will be discussed in Example~\ref{star_example}.

\subsection{Flats}

A \emph{flat} of an hyperplane arrangement $\ha$ is a subspace of the ambient space obtained by intersecting a subset of hyperplanes of $\ha$. Let $\Pi[\ha]$ denote the set of flats. Then it is a graded poset under inclusion, with the center as the minimum element (the intersection of all flats), and the ambient space as the maximum element (the empty intersection of flats). Note, the center is the only subset in $\ha$ which is both a face and a flat. 

The \emph{support} of a face $F$ is the smallest flat which contains $F$, and we denote it as $\supp(F)$. It is the intersection of all flats which contain $F$.

\subsection{Tits Product}

For a hyperplane $H$, we denote its two associated half-spaces by $H^+$ and $H^-$. The choice can be arbitrary but it should be fixed. We let $H^0 = H$ and observe that $H^0 = H^+ \cap H^-.$ In this notation, a face of $\ha$ is a subset of the ambient space of the form $$F = \bigcap_{i\in \mathbb{H}}H_i^{\epsilon_i}$$ where $\mathbb{H}$ is the set of hyperplanes, and $\epsilon_i\in \{+,-,0\}$.

Though different choices of $\epsilon_i$'s may yield the same face,  there is a canonical way to write $F$ in this form, that is,
$$F = \bigcap_{i\in \mathbb{H}}H_i^{\epsilon_i(F)}$$ where $$\epsilon_i(F) = \begin{cases}
    0 & \text{if } F\in H_i, \\
    + & \text{if } \interior(F) \in \interior(H_i^+),\\
    - & \text{if } \interior(F) \in \interior(H_i^-).
\end{cases}$$

For faces $F$ and $G$, define the face $FG$ by 
$$\epsilon_i(FG) := \begin{cases}
    \epsilon_i(F) & \text{if } \epsilon_i(F) \neq 0, \\ 
    \epsilon_i(G) & \text{if }\epsilon_i(F) = 0.
\end{cases}$$

We refer to $FG$ as the \emph{Tits product} of $F$ and $G$. The product has a geometric meaning: if we move from an interior point of $F$ to an interior point of $G$ along a straight line, then $FG$ is the face that we are in after moving a small positive distance. Note, the Tits product is associative, and it admits an identity element, which is the central face. Hence $\Sigma[\ha]$ admits a monoid structure with respect to the Tits product, and we call this the \emph{Tits monoid.}

An example of Tits product will be discussed in Example \ref{Tits_product_example}.

\subsection{Cones and Top-cones}

A \emph{cone} of an arrangement $\ha$ is a subset of the ambient space which can be obtained by intersecting some subset of half-spaces in the arrangement.

Note that if a point lies in $V$, then the face that the point lies in also lies in $V$. Thus every cone is a union of faces. Also observe that if faces $F$ and $G$ lie in $V$, then $FG$ also lies in $V$. By multiplying all the faces lying in $V$ in different orders we obtain the ``largest" faces lying in $V$. A cone of maximum dimension is called a \emph{top-cone}.

\begin{proposition}\cite[Proposition 2.2]{MS}
    For a cone $V$, the following are equivalent.
\begin{enumerate}
    \item $V$ is a top-cone.
    \item $V$ contains at least one chamber.
\end{enumerate}
\end{proposition}

\subsection{Braid Arrangement of a Finite Set}

The \emph{braid arrangement} of $[n]$ consists of $\binom{n}{2}$ hyperplanes in $\mathbb{R}^n$ which are defined by $$x_i = x_j$$ for $1\leq i < j < n.$ 

The canonical linear order of the set $[n]$ is not relevant to the definition of the arrangement. Hence we may proceed as follows. Let $I$ be a finite set. The \emph{braid arrangement} of $I$, denoted as $B_I$, consists of $\binom{|I|}{2}$ hyperplanes which are defined by $$x_a = x_b$$ in $\mathbb{R}^I$, for $a\neq b \in I$. Its Coxeter Group is the group of bijections from $I$ to itself.

Note that for any hyperplane arrangement, one can associate a variety of geometric notions such as faces, flats, and more. In the case of the braid arrangement, these geometric notions correspond to well-known combinatorial objects, as described in \cite[Table 6.2]{MS}. We outline the bijections that appear in the following sections below.

\begin{center}
    face $\longleftrightarrow$ set composition \\
    chamber $\longleftrightarrow$ linear order \\
    flat $\longleftrightarrow$ set partition \\
    cone $\longleftrightarrow$ preorder \\
    top cone $\longleftrightarrow$ partial order
\end{center}

\begin{example} \label{B(I)}
    Figure \ref{B(I)_color} is the same picture as in \cite[p.143]{MS} but we change their labels $\{a,b,c,d\}$ to $\{x,y,z,w\}$ to prevent confusion in discussions of $ab$ and $cd$-indices. The braid arrangement on $I=\{x,y,z,w\}$ consists of 6 hyperplanes $\{m = n\mid m \neq n \in \{x,y,z,w\}\}$.

\begin{figure}[h!] 
\begin{center}

  \includegraphics[width=0.8\textwidth]{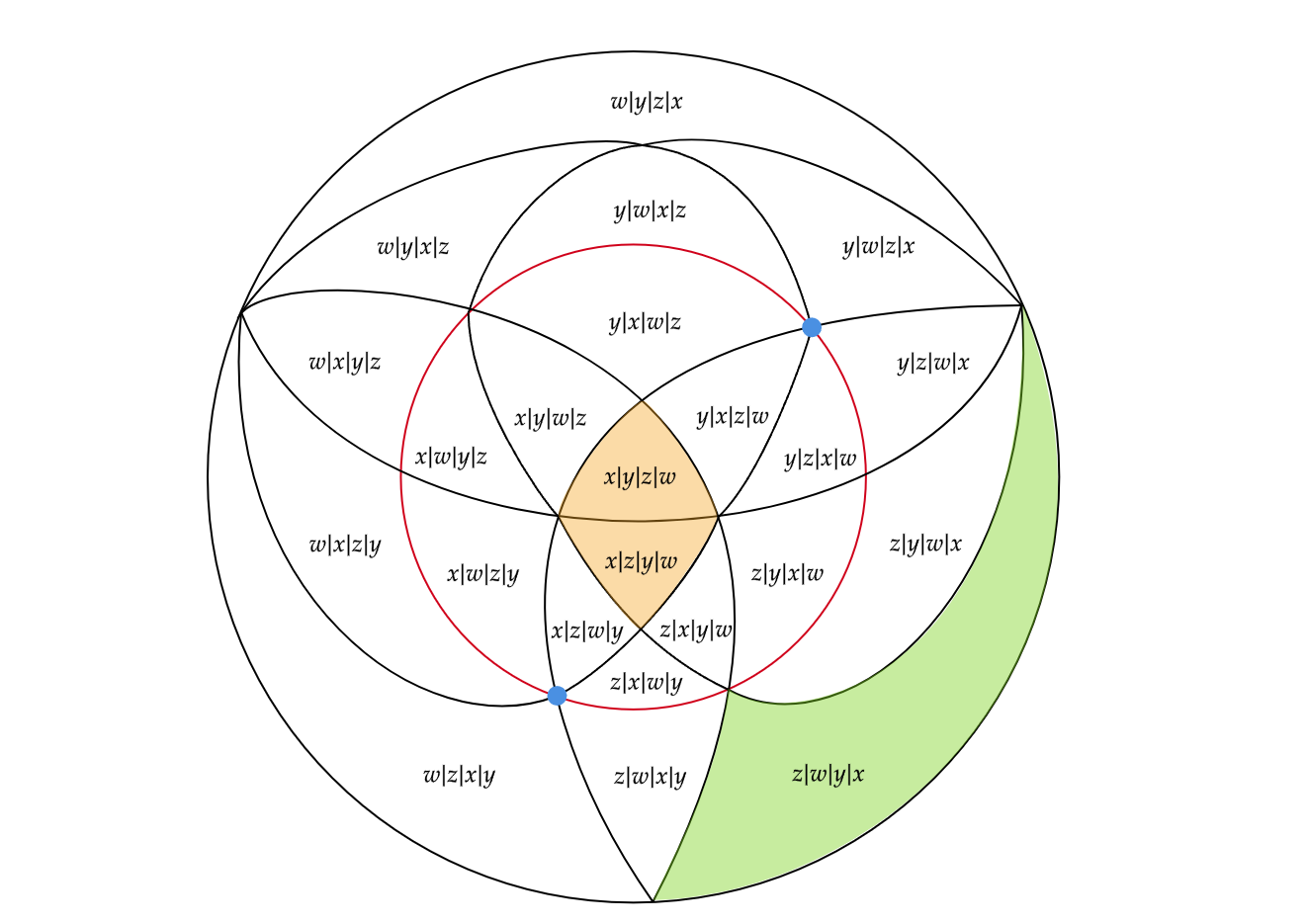}  
\caption{the braid arrangement on $I = \{x,y,z,w\}$}\label{B(I)_color}
\end{center}
\end{figure}

The green chamber labeled $z|w|y|x$ corresponds to the linear order $z<w<y<x$. The top cone (cone containing at least one chamber) in orange corresponds to the partial order on $\{x,y,z,w\}$ defined by relations $x<y$, $x<z$, $x<w$, $y<w$, $z<w$. The hyperplane in red ($x=w$) is also a flat corresponding to the set partition $\{xw,y,z\}$ with maximal faces permutations of $xw|y|z$. The flat represented by the two blue points corresponds to the set partition $\{xzw,y\}$ with maximal faces permutations of $xzw|y$.
\end{example}

\begin{example}\label{star_example}
In Figure \ref{B(I)_color}, consider the face corresponding to the set composition $x|yz|w$, then $$\starr(x|yz|w) = \{x|yz|w, x|y|z|w, x|z|y|w\}.$$ $$\mathrm{top}\text{-}\mathrm{star}(x|yz|w) = \{x|y|z|w, x|z|y|w\}$$
\end{example}

\begin{example}\label{Tits_product_example}
In Figure \ref{B(I)_color}, consider the faces $F$, $G$, corresponding to the set composition $yz|x|w$, and $z|w|y|x$, respectively, then the Tits product $FG$ is the chamber $z|y|x|w$.
\end{example}

We can apply those geometric descriptions on enumeration of (labeled) objects of Hopf monoids.

\section{The Hopf monoid of convex geometries and the invariants}\label{convex_geometries}

\subsection{Introduction}

Let $I$ be a finite set. A \emph{partial order} on $I$ is a relation on $I$ that is reflexive, antisymmetric, and transitive. For a partial order $p$ on $I$ and elements $a,b\in I$, $a$ is related to $b$ by $p$ is written as $a\leq_p b$. We may omit $p$ and write $a\leq b$ if there is no confusion of $p$. A partial order with no relations is called an \emph{antichain}.

A subset $S\subseteq I$ is called a \emph{lower set} of $p$ if for any $s\in S$, $a\in I$ with $a\leq_p s$, then $a\in S$. Specifically,  $\emptyset$ is a lower set for any partial order. 

Suppose $I = S\sqcup T$. Given two partial order $p_1$ on $S$, and $p_2$ on $T$, the \emph{parallel composition} $p_1\parallel p_2$ is the partial order on $I$ in which there are no relations between $S$ and T, while these sets are ordered according to $p_1$ and $p_2$, respectively. Given a partial order $p$ on $I$, the \emph{restriction} of $p$ on $S$, namely $p|_S$, is a partial order on $S$ where the relations are induced from $p$.

\begin{example}
Let $I = \{a,b,c,d,e\}$, $S = \{a,b,c\}$, $T = \{d,e\}$. Let $p_1$, $p_2$, $p$ be partial orders on $S$, $T$, $I$, respectively, with the following Hasse diagrams.

\tikzset{every picture/.style={line width=0.75pt}}    

\begin{center}

\begin{tikzpicture}[x=0.6pt,y=0.6pt,yscale=-1,xscale=1]

\draw    (498.75,151.67) -- (535,219.67) ;
 
\draw    (570.75,151.67) -- (535,219.67) ;
 
\draw    (534.5,83.67) -- (570.75,151.67) ;
 
\draw    (534.5,83.67) -- (498.75,151.67) ;

\draw    (606.5,83.67) -- (570.75,151.67) ;

\draw  [fill={rgb, 255:red, 0; green, 0; blue, 0 }  ,fill opacity=1 ] (533.17,218.83) .. controls (533.17,217.82) and (533.99,217) .. (535,217) .. controls (536.01,217) and (536.83,217.82) .. (536.83,218.83) .. controls (536.83,219.85) and (536.01,220.67) .. (535,220.67) .. controls (533.99,220.67) and (533.17,219.85) .. (533.17,218.83) -- cycle ;

\draw  [fill={rgb, 255:red, 0; green, 0; blue, 0 }  ,fill opacity=1 ] (496.92,151.67) .. controls (496.92,150.65) and (497.74,149.83) .. (498.75,149.83) .. controls (499.76,149.83) and (500.58,150.65) .. (500.58,151.67) .. controls (500.58,152.68) and (499.76,153.5) .. (498.75,153.5) .. controls (497.74,153.5) and (496.92,152.68) .. (496.92,151.67) -- cycle ;

\draw  [fill={rgb, 255:red, 0; green, 0; blue, 0 }  ,fill opacity=1 ] (568.92,151.67) .. controls (568.92,150.65) and (569.74,149.83) .. (570.75,149.83) .. controls (571.76,149.83) and (572.58,150.65) .. (572.58,151.67) .. controls (572.58,152.68) and (571.76,153.5) .. (570.75,153.5) .. controls (569.74,153.5) and (568.92,152.68) .. (568.92,151.67) -- cycle ;

\draw  [fill={rgb, 255:red, 0; green, 0; blue, 0 }  ,fill opacity=1 ] (604.83,83.67) .. controls (604.83,82.65) and (605.65,81.83) .. (606.67,81.83) .. controls (607.68,81.83) and (608.5,82.65) .. (608.5,83.67) .. controls (608.5,84.68) and (607.68,85.5) .. (606.67,85.5) .. controls (605.65,85.5) and (604.83,84.68) .. (604.83,83.67) -- cycle ;

\draw  [fill={rgb, 255:red, 0; green, 0; blue, 0 }  ,fill opacity=1 ] (532.67,83.67) .. controls (532.67,82.65) and (533.49,81.83) .. (534.5,81.83) .. controls (535.51,81.83) and (536.33,82.65) .. (536.33,83.67) .. controls (536.33,84.68) and (535.51,85.5) .. (534.5,85.5) .. controls (533.49,85.5) and (532.67,84.68) .. (532.67,83.67) -- cycle ;

\draw    (77.75,137.67) -- (114,205.67) ;

\draw    (149.75,137.67) -- (114,205.67) ;

\draw  [fill={rgb, 255:red, 0; green, 0; blue, 0 }  ,fill opacity=1 ] (112.17,204.83) .. controls (112.17,203.82) and (112.99,203) .. (114,203) .. controls (115.01,203) and (115.83,203.82) .. (115.83,204.83) .. controls (115.83,205.85) and (115.01,206.67) .. (114,206.67) .. controls (112.99,206.67) and (112.17,205.85) .. (112.17,204.83) -- cycle ;

\draw  [fill={rgb, 255:red, 0; green, 0; blue, 0 }  ,fill opacity=1 ] (75.92,137.67) .. controls (75.92,136.65) and (76.74,135.83) .. (77.75,135.83) .. controls (78.76,135.83) and (79.58,136.65) .. (79.58,137.67) .. controls (79.58,138.68) and (78.76,139.5) .. (77.75,139.5) .. controls (76.74,139.5) and (75.92,138.68) .. (75.92,137.67) -- cycle ;

\draw  [fill={rgb, 255:red, 0; green, 0; blue, 0 }  ,fill opacity=1 ] (147.92,137.67) .. controls (147.92,136.65) and (148.74,135.83) .. (149.75,135.83) .. controls (150.76,135.83) and (151.58,136.65) .. (151.58,137.67) .. controls (151.58,138.68) and (150.76,139.5) .. (149.75,139.5) .. controls (148.74,139.5) and (147.92,138.68) .. (147.92,137.67) -- cycle ;

\draw    (315,127) -- (315,215) ;

\draw  [fill={rgb, 255:red, 0; green, 0; blue, 0 }  ,fill opacity=1 ] (313.17,128.83) .. controls (313.17,127.82) and (313.99,127) .. (315,127) .. controls (316.01,127) and (316.83,127.82) .. (316.83,128.83) .. controls (316.83,129.85) and (316.01,130.67) .. (315,130.67) .. controls (313.99,130.67) and (313.17,129.85) .. (313.17,128.83) -- cycle ;

\draw  [fill={rgb, 255:red, 0; green, 0; blue, 0 }  ,fill opacity=1 ] (313.17,215) .. controls (313.17,213.99) and (313.99,213.17) .. (315,213.17) .. controls (316.01,213.17) and (316.83,213.99) .. (316.83,215) .. controls (316.83,216.01) and (316.01,216.83) .. (315,216.83) .. controls (313.99,216.83) and (313.17,216.01) .. (313.17,215) -- cycle ;

\draw (542.29,210.24) node [anchor=north west][inner sep=0.75pt]   [align=left] {$\displaystyle a$};

\draw (577.43,142.1) node [anchor=north west][inner sep=0.75pt]   [align=left] {$\displaystyle d$};

\draw (483.29,142.24) node [anchor=north west][inner sep=0.75pt]   [align=left] {$\displaystyle b$};

\draw (539,67.67) node [anchor=north west][inner sep=0.75pt]   [align=left] {$\displaystyle c$};

\draw (611.43,67.1) node [anchor=north west][inner sep=0.75pt]   [align=left] {$\displaystyle e$};

\draw (121.29,196.24) node [anchor=north west][inner sep=0.75pt]   [align=left] {$\displaystyle a$};

\draw (156.43,128.1) node [anchor=north west][inner sep=0.75pt]   [align=left] {$\displaystyle c$};

\draw (62.29,128.24) node [anchor=north west][inner sep=0.75pt]   [align=left] {$\displaystyle b$};

\draw (327.29,205.24) node [anchor=north west][inner sep=0.75pt]   [align=left] {$\displaystyle d$};

\draw (329.43,116.1) node [anchor=north west][inner sep=0.75pt]   [align=left] {$\displaystyle e$};

\draw (105,246) node [anchor=north west][inner sep=0.75pt]   [align=left] {$\displaystyle p_{1}$};

\draw (309,246) node [anchor=north west][inner sep=0.75pt]   [align=left] {$\displaystyle p_{2}$};

\draw (531.5,245.5) node [anchor=north west][inner sep=0.75pt]   [align=left] {$\displaystyle p$};

\end{tikzpicture}
\end{center}

Then we have the Hasse diagrams of $p_1 \parallel p_2$ and $p|_S$ as follows.
\begin{center}
 
\tikzset{every picture/.style={line width=0.75pt}} 

\begin{tikzpicture}[x=0.6pt,y=0.6pt,yscale=-1,xscale=1]

\draw    (135.75,71.67) -- (172,139.67) ;

\draw    (207.75,71.67) -- (172,139.67) ;

\draw  [fill={rgb, 255:red, 0; green, 0; blue, 0 }  ,fill opacity=1 ] (170.17,138.83) .. controls (170.17,137.82) and (170.99,137) .. (172,137) .. controls (173.01,137) and (173.83,137.82) .. (173.83,138.83) .. controls (173.83,139.85) and (173.01,140.67) .. (172,140.67) .. controls (170.99,140.67) and (170.17,139.85) .. (170.17,138.83) -- cycle ;

\draw  [fill={rgb, 255:red, 0; green, 0; blue, 0 }  ,fill opacity=1 ] (133.92,71.67) .. controls (133.92,70.65) and (134.74,69.83) .. (135.75,69.83) .. controls (136.76,69.83) and (137.58,70.65) .. (137.58,71.67) .. controls (137.58,72.68) and (136.76,73.5) .. (135.75,73.5) .. controls (134.74,73.5) and (133.92,72.68) .. (133.92,71.67) -- cycle ;

\draw  [fill={rgb, 255:red, 0; green, 0; blue, 0 }  ,fill opacity=1 ] (205.92,71.67) .. controls (205.92,70.65) and (206.74,69.83) .. (207.75,69.83) .. controls (208.76,69.83) and (209.58,70.65) .. (209.58,71.67) .. controls (209.58,72.68) and (208.76,73.5) .. (207.75,73.5) .. controls (206.74,73.5) and (205.92,72.68) .. (205.92,71.67) -- cycle ;

\draw    (270,60) -- (270,148) ;

\draw  [fill={rgb, 255:red, 0; green, 0; blue, 0 }  ,fill opacity=1 ] (268.17,61.83) .. controls (268.17,60.82) and (268.99,60) .. (270,60) .. controls (271.01,60) and (271.83,60.82) .. (271.83,61.83) .. controls (271.83,62.85) and (271.01,63.67) .. (270,63.67) .. controls (268.99,63.67) and (268.17,62.85) .. (268.17,61.83) -- cycle ;

\draw  [fill={rgb, 255:red, 0; green, 0; blue, 0 }  ,fill opacity=1 ] (268.17,148) .. controls (268.17,146.99) and (268.99,146.17) .. (270,146.17) .. controls (271.01,146.17) and (271.83,146.99) .. (271.83,148) .. controls (271.83,149.01) and (271.01,149.83) .. (270,149.83) .. controls (268.99,149.83) and (268.17,149.01) .. (268.17,148) -- cycle ;

\draw    (473.51,88.26) -- (473.51,152.52) ;

\draw  [fill={rgb, 255:red, 0; green, 0; blue, 0 }  ,fill opacity=1 ] (472.17,89.6) .. controls (472.17,88.86) and (472.77,88.26) .. (473.51,88.26) .. controls (474.25,88.26) and (474.84,88.86) .. (474.84,89.6) .. controls (474.84,90.33) and (474.25,90.93) .. (473.51,90.93) .. controls (472.77,90.93) and (472.17,90.33) .. (472.17,89.6) -- cycle ;
 
\draw  [fill={rgb, 255:red, 0; green, 0; blue, 0 }  ,fill opacity=1 ] (472.17,152.52) .. controls (472.17,151.78) and (472.77,151.18) .. (473.51,151.18) .. controls (474.25,151.18) and (474.84,151.78) .. (474.84,152.52) .. controls (474.84,153.26) and (474.25,153.86) .. (473.51,153.86) .. controls (472.77,153.86) and (472.17,153.26) .. (472.17,152.52) -- cycle ;

\draw    (473.51,26.67) -- (473.51,90.93) ;

\draw  [fill={rgb, 255:red, 0; green, 0; blue, 0 }  ,fill opacity=1 ] (471.67,24.83) .. controls (471.67,23.82) and (472.49,23) .. (473.51,23) .. controls (474.52,23) and (475.34,23.82) .. (475.34,24.83) .. controls (475.34,25.85) and (474.52,26.67) .. (473.51,26.67) .. controls (472.49,26.67) and (471.67,25.85) .. (471.67,24.83) -- cycle ;

\draw (179.29,130.24) node [anchor=north west][inner sep=0.75pt]   [align=left] {$\displaystyle a$};

\draw (214.43,62.1) node [anchor=north west][inner sep=0.75pt]   [align=left] {$\displaystyle c$};

\draw (120.29,62.24) node [anchor=north west][inner sep=0.75pt]   [align=left] {$\displaystyle b$};

\draw (282.29,138.24) node [anchor=north west][inner sep=0.75pt]   [align=left] {$\displaystyle d$};
 
\draw (284.43,49.1) node [anchor=north west][inner sep=0.75pt]   [align=left] {$\displaystyle e$};
 
\draw (193,198) node [anchor=north west][inner sep=0.75pt]   [align=left] {$\displaystyle p_{1} \parallel \ p_{2}$};
 
\draw (480.99,142.97) node [anchor=north west][inner sep=0.75pt]   [align=left] {$\displaystyle a$};
 
\draw (482.56,77.87) node [anchor=north west][inner sep=0.75pt]   [align=left] {$\displaystyle b$};
 
\draw (483.43,14.1) node [anchor=north west][inner sep=0.75pt]   [align=left] {$\displaystyle c$};
 
\draw (462,197) node [anchor=north west][inner sep=0.75pt]   [align=left] {$\displaystyle p|_{S}$};

\end{tikzpicture}
   
\end{center}

\end{example}

The structure of the Hopf monoid of partial orders was discussed by Aguiar and Bastidas in \cite{ABV}. Specifically, they discussed two canonical bases related by M\"obius inversion and the corresponding dual bases. We focus on the $\mathtt{G}$ basis in their discussion. This is also the basis discussed by Aguiar and Mahajan in \cite{MR2724388} when they described the Hopf monoid of \emph{closure operators}, which we will look into when we discuss \emph{convex geometries}.

For each finite set $I$, let $\po[I]$ be the vector space with basis the set of all partial orders on $I$. Given $I=S\sqcup T$ (we allow $S$ and $T$ to be the empty set), partial orders $p$ on $S$, $q$ on $T$, and $r$ on $I$, let $\mu_{S,T}({p_1}, {p_2})= p_1 \parallel p_2$,

$$\Delta_{S,T}(p) = 
\begin{cases} 
{p|_S}\otimes {p|_T} & \text{if $S$ is a \emph{lower set} of $p$}, \\ 
0 & \text{otherwise.}
\end{cases}$$

In the above definition, $S$ is a \emph{lower set} of $p$ if for any $s\in S$, $a\in I$ with $a\leq_p s$, then $a\in S$. Extend $\mu_{S,T}$ and $\Delta_{S,T}$ linearly, then $\po$ equipped with $\mu$ and $\Delta$ is a Hopf monoid. In \cite[Section 13.9.5]{MR2724388} it was mentioned that there is a Hopf monoid structure of convex geometries which admits $\po$ as a Hopf submonoid. We give a discussion on it as follows.

Let $I$ be a finite set, and let $2^I$ denote its power set. A \emph{closure operator} on $I$ is a map $$ c: 2^I \rightarrow 2^I$$ such that for every $A$, $B\in 2^I$, \begin{itemize}
    \item $c(A) = c^2(A)$.
    \item $A \subseteq c(A)$.
    \item If $A\subseteq B$, then $c(A)\subseteq c(B)$.
\end{itemize}

The subset $S\subseteq I$ with $c(S) = S$ is called a \emph{closed set}. In the following content, if we write $cl = \{S_1,...,S_k\}$, then this specifies $cl$ by the associated closed sets $S_1....,S_k$. A closure operator is called \emph{loopless} if $c(\emptyset) = \emptyset$. In the setting of Hopf monoids, the closure operators are labeled. From now on we assume all closure operators that we discuss are loopless.

We have the notions of merging and breaking on closure operators. For $c$, $d$ closure operators on $S$, $T$, respectively, and $e$ a closure operator on $I = S\sqcup T$, we have the \emph{direct sum}, $c\oplus d$ such that for $A\subseteq I$, $$(c\oplus d)(A) = c(A\cap S) \cup d(A\cap T).$$ We have \emph{restriction} of $c$ on $S$ and \emph{contraction} of $c$ through $S$ such that for $A\subseteq S$, $B\subseteq T$, $$(e|_S)(A) = e(A), \text{ } (e/_S)(B) = e(S\cup B)\cap T.$$

Consider $\lc$ such that for each finite set $I$, $\lc[I]$ is the vector space spanned by all loopless closure operators. Define $\mu$, $\Delta$ on $\lc$ as follows. For $c$, $d$ closure operators on $S$, $T$, respectively, and $e$ a closure operator on $I = S\sqcup T$, $$\mu_{S,T}(c,d) = c\oplus d.$$
$$\Delta_{S,T}(e) = 
\begin{cases} 
{e|_S}\otimes {e/_S} & \text{if $S$ is closed}, \\ 
0 & \text{otherwise.}
\end{cases}$$

Then $\lc$ is a Hopf monoid \cite[Section 13.8]{MR2724388}.

A \emph{convex geometry} with ground set $I$ is a closure operator $g: 2^I \rightarrow 2^I$ that satisfies the \emph{anti-exchange axiom}. That is, if $a\in g(A\cup \{b\})$, $a\neq b$, and $a$, $b \notin g(A)$, then $b \notin g(A\cup \{a\})$ for every $A \in 2^I$ and $a$, $b \in I$. If a subset $A\subseteq I$ is closed under $g$, then we say $A$ is a \emph{convex set}.

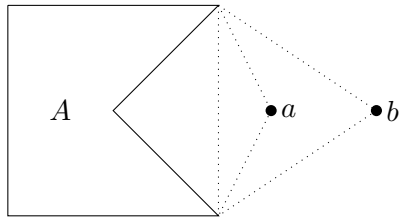
\begin{figure}[h]
    \centering
\begin{tikzpicture}[scale=0.7]
\draw (-2,0) -- (2,0) -- (0,-2) -- (2,-4) --(-2,-4) -- (-2,0) (-1,-2) node{$A$} (3,-2) node[right]{$a$} (5,-2) node[right]{$b$};
\draw[dotted] (2,0) -- (2,-4)  (2,0)--(3,-2) (2,0) -- (5,-2) (3,-2) -- (2,-4) (5,-2) -- (2,-4);
\fill (3,-2) {circle(3pt)} (5,-2) {circle(3pt)};
\end{tikzpicture}
    \caption{Convex geometry induced from Euclidean closure}
\end{figure}

The notion of convex geometries generalizes partial orders in sense of order ideals, because given a partial order $p$ on ground set $I$, we can define a convex geometry $g_p$ such that for $A\subseteq I$ 

\begin{equation}\label{partial_order_to_convex_geometries}
    g_p(A) = \{x \in I \mid \text{$x\leq a$ in $p$  for some $a\in A$} \}.
\end{equation}

Also, the species $\lcg$ with $\lcg[I]$ the vector space spanned by all loopless convex geometries, together with $\mu$, $\Delta$ described the same as those of $\lc$, carries a Hopf monoid structure. Indeed we have the following injective morphisms of Hopf monoids.

\begin{equation}\label{injections}
  \po \hookrightarrow \lcg \hookrightarrow \lc.  
\end{equation}

\subsection{Canonical characters and polynomial invariants}\label{characters_polynomial_invariants}

On $\lc$, we define four characters $\eta$, $\zeta$, $\varphi$, and $\varphi'$. Specifically, define $\eta$ with $\eta_I(c) = 1$ for all basis element $c\in \lc[I]$. We then set $\zeta= \overline{\eta}^{-1}$, $\varphi = \zeta\ast\eta$, and $\varphi' = \eta \ast \zeta.$ Because of the injective morphisms (\ref{injections}), these four characters are well-defined on $\lcg$ and $\po$. We refer to \cite{ABV} for detailed discussions on these characters on the Hopf monoid of partial orders $\po$, and we review some of their discussions before moving to convex geometries. Specifically, on $\po$,  $\zeta: \po \rightarrow \KE$ was computed as \begin{equation*}
    \zeta_I(p) = \begin{cases} 1 & \text{ if $p$ is the antichain on $I$}, \\0 & \text{otherwise} \end{cases}
\end{equation*}
for basis element $p\in\po[I]$. The polynomial invariant $\chi^\eta_I(n)$ is the \emph{order polynomial}, which was introduced by Stanley \cite[Definition (i)]{MR0269545} and counts the number of $p$ order-preserving maps $f: I\rightarrow [n]$. The polynomial invariant $\chi^\zeta_I(p)(n)$ counts the number of strict $p$ order-preserving maps $I\rightarrow [n]$, so it corresponds to the \emph{strict order polynomial}.

One interesting application of the reciprocity result in view of characters is that we can obtain the Stanley's reciprocity \cite[Theorem 3]{MR0269545}. Specifically, by the fact that $\zeta = \overline{\eta}^{-1}$, we have $$(-1)^{|I|}\chi^\eta_I(p)(-n) = \chi^\zeta_I(p)(n).$$

For $p\in\po[I]$, let $\Min(p)$ be the subset of $I$ consisting of minimal elements of $p$, and let $\Max(p)$ be the subset of $I$ consisting of maximal elements of $p$. Let $\min(p) = |\Min(p)|$ and $\max(p) = |\Max(p)|.$ Then we have $$\varphi_I(p) = 2^{\min(p)}, \text{ }\varphi'_I(p) = 2^{\max(p)}.$$

Let $\llbracket n \rrbracket = [n]\sqcup [\overline{n}]$ equipped with the linear order $\overline{1} < 1 < \overline{2}<2<...<\overline{n}<n.$ Stembridge introduced the \emph{enriched order polynomial} of a partial order $p$ \cite[Section 4]{Stembridge}, which counts the number of functions $f: I\rightarrow \llbracket n \rrbracket$ that are order preserving, and in addition satisfies that $f^{-1}(\overline{i})$ is an antichain for each $i = 1,2,...,n$. 

Indeed, we have $\chi^\varphi_I(p)$ is the enriched order polynomial of $p$, and $\chi^{\varphi'}_p$ counts the number of functions $f: I\rightarrow \llbracket n \rrbracket$ that are order preserving, and in addition satisfies that $f^{-1}(i)$ is an antichain for each $i = 1,2,...,n$. Since $\varphi$ and $\varphi'$ are odd, we can use the self-reciprocity properties of $\chi^\varphi_I$ and $\chi^{\varphi'}$ to obtain the Stembridge's reciprocity \cite[Proposition 4.2]{Stembridge}: $$\chi^{\varphi}_I(p)(-n) = (-1)^{|I|}\chi^{\varphi}_I(p)(n), \text{ }\chi^{\varphi'}_I(p)(-n) = (-1)^{|I|}\chi^{\varphi'}_I(p)(n).$$

In the following discussion, we extend the discussion of these characters to the setting of convex geometries. We demonstrate that the associated polynomial invariants generalize the order and enriched order polynomials, and we prove the reciprocity results for them in a unified framework.

Recall for convex geometry $g$ on ground set $I$, $\eta_I(g) = 1$ . We say that $g$ is \emph{discrete} if  $g(A) = A$ for all $A \in 2^{I}$. Then we have $\overline{\eta}^{-1} := \zeta: \klcgd \rightarrow \KE$ as follows. 
 
 \begin{proposition}\label{inverse_eta_bar_cg}
     $\zeta_I(g) = \begin{cases}1 & \text{if }g\text{ is discrete}, \\ 0 & \text{otherwise}. \end{cases}$
 \end{proposition}
\begin{proof}
    Let $K\subseteq I$ be convex. Denote the set of \emph{extreme points} of $K$ as $$\Ex(K) := \{a\in K \mid a\notin g({K\backslash a})\} = \{a\in K \mid K \backslash a \text{ is convex}\}.$$
Then it is easy to see that for $I = S\sqcup T$, 
    $S\subseteq I \text{ is convex and } g/_S \text{ is discrete iff }T\subseteq \Ex(I)$.

Computing the convolution product of $\eta$ and $\overline{\zeta}$, we have $$(\eta*\overline{\zeta})_I(g) = \sum_{I = S\sqcup T,\text{ }\text{$S$ is convex}}\eta_S({g|_S})\overline{\zeta}_T({g/_S})$$ $$ \hspace{10mm}= \sum_{I = S\sqcup T,\text{ } \text{$S$ is convex, $g/_S$ is discrete}}(-1)^{|T|} = \sum_{I = S\sqcup T,\text{ } T\subseteq \ex(I)} (-1)^{|T|}. \hspace{59mm}$$  This is the sum of the M\"obius function over the Boolean poset of $\Ex(I)$. Hence we have $$(\eta * \overline{\zeta})_I(g)\begin{cases}1 & \text{ if $\Ex(I) = \varnothing$,} \\ 0 & \text{ otherwise.} \end{cases}$$

We claim that $\Ex(I) = \varnothing \Leftrightarrow I = \varnothing$. The backward direction is clear. So suppose $I\neq \varnothing$ but $\ex(I) = \varnothing$. That is, for all $S\subseteq I$ with $|S| = |I|-1$, $g(S) = I$, so there is no convex subset of $I$ with size $|I|-1$. But this contradicts \cite[theorem 2.1 (b)]{EJ}, since we can obtain a size $|I|-1$ convex set by constructing a flag of convex sets $\varnothing = S_0 \subseteq S_1\subseteq ...\subseteq S_{|I|-1} \subseteq S_{|I|} = I$ with $|S_i| = i$. Hence we conclude that $\overline{\zeta}$ is the convolution inverse of $\eta$.
\end{proof}

Let $g$ be a convex geometry on $I$. Let $f: I\rightarrow [n]$. For a convex set $K$, let $f^K = \max_{x\in K}f(x)$.
 $f$ is called \emph{extremal} \cite[p.260]{EJ} if for each convex set $K\subseteq I$, $$\{x\in K \mid f(x) = f^K\} \cap \Ex(K) \neq \varnothing.$$
$f$ is called \emph{strictly extremal}\cite[p.260]{EJ} if for each convex set $K$, $$\{x\in K \mid f(x) = f^K\} \subseteq \Ex(K).$$ 

Note that when $g = g_p$ as in (\ref{partial_order_to_convex_geometries}), the corresponding extremal functions are precisely $p$ order preserving functions, and the corresponding strictly extremal functions are strict $p$ order preserving functions.

To describe the polynomial invariant associated with $\eta$ and $\zeta$, we define another class of functions of $f: I\rightarrow [n]$ and show an equivalence condition. Specifically, we say that $f$ is \emph{convex} if for any convex set $S$ on $[n]$, $f^{-1}(S)$ is convex in $g$. We say that $f$ is \emph{strictly convex} if $f$ is convex and $g_{f^{-1}[i]:f^{-1}[i+1]}$ is discrete for $1\leq i \leq n-1$.

\begin{lemma}\label{lemma1eq}
For $A\subseteq I$, $\Ex(g({A})) \subseteq A$. 
\end{lemma}
\begin{proof}

    If $a\notin A$, then $A\backslash a = A\Rightarrow g({A\backslash a}) = g(A)$. Then $g(g(A)\backslash a) = g(A)$, and so $a\notin \Ex(g(A))$.
\end{proof} 

\begin{lemma}\label{lemma2eq}\textnormal{\cite[Theorem 2.1]{EJ}}
For convex set $K\subseteq I$, $K = g(\Ex(K))$.
\end{lemma}

\begin{proposition}\label{extremal=convex}
$f$ is extremal iff $f$ is convex. $f$ is strictly extremal iff $f$ is strictly convex.
\end{proposition}
\begin{proof}
We show the first equivalence statement.
    
``$\Rightarrow$": suppose $f$ is extremal. Let $A$ be a convex set in $[n]$. By previous discussions, $A = [m]$ for some $m \leq n$. Consider $f^{-1}(A)$, and let $K = g(f^{-1}(A))$.

If $K \neq f^{-1}(A)$, by assumption there exists $a\in \Ex(K)$ such that $f(a) >m$. By Lemma \ref{lemma1eq}, $\Ex(K)\subseteq f^{-1}(A)$ and hence $ a\in f^{-1}(A)$. This contradicts $a\notin f^{-1}(A)$ since $f(a) > m$. So $f^{-1}(A)$ is convex.

``$\Leftarrow$": let $W\subseteq I$ be a convex set, let $i = f^W$. We need to show that $$\{x\in W: f(x) = i\} \cap \Ex(W)\neq \varnothing. $$

Suppose the intersection above is empty, i.e., for all $x\in \Ex(W)$, $f(x) < i$. Then $\Ex(W)\subseteq f^{-1}([i-1]) \Rightarrow g(\Ex(W)) \subseteq g(f^{-1}([i-1])) = f^{-1}([i-1])$. But by Lemma \ref{lemma2eq} we have $W = g(\Ex(W))$, so $W \subseteq f^{-1}([i-1])$, contradicting the fact that $f^W = i$.

The second statement follows from the first statement, the definition of strictly extremal functions, and the assumption $g_{f^{-1}[i]:f^{-1}[i+1]}$ is discrete for $1\leq i \leq n-1$.
\end{proof}

\begin{lemma}\label{chainconvex}
Let $g$ be a convex geometry on $I$. Then for $\varnothing \subseteq K_1\subseteq K_2 \subseteq I$ such that $K_1$ is convex, $$K_2\text{ is convex in $g$}\Leftrightarrow K_2\backslash K_1 \text{ is convex in }g/_{K_1}. $$
\end{lemma}
\begin{proof}
``$\Rightarrow$": $g(K_2) = K_2 \Rightarrow g/_{K_1}(K_2\backslash K_1) = g(K_2)\backslash K_1 = K_2\backslash K_1$.

`` $\Leftarrow$": $g/_{K_1}(K_2\backslash K_1)= K_2\backslash K_1 \Rightarrow g((K_2 \backslash K_1)\cup K_1) \backslash K_1 = K_2 \backslash K_1 \Rightarrow g(K_2)\backslash K_1 = K_2 \backslash K_1$. Since $K_2 \subseteq g(K_2)$, if $K_2 \subsetneq g(K_2)$, since $K_1\subseteq K_2$, we have $g(K_2)\backslash K_1 \supsetneq K_2 \backslash K_1$, which is a contradiction. Hence $g(K_2) = K_2$.
\end{proof}

Let $\chi^\eta$ be the polynomial invariant associated with $\eta$ and $\chi^\zeta$ be the polynomial invariant associated with $\zeta$. Then we have the following combinatorial descriptions.

\begin{theorem}\label{poly_cg_eta}
 $\chi^\eta_I(g)(n)$ counts the number of extremal functions $f: I\rightarrow [n]$.
$\chi^\zeta_I(g)(n)$ counts the number of strictly extremal functions $f: I \rightarrow [n]$.

\end{theorem}
\begin{proof}
 By Lemma \ref{chainconvex}, we have the polynomial invariant associated with $\eta$ is as follows: for $g\in \overline{\bf cG}[I]$, $$\chi^\eta_I(g)(n) = \sum_{I = S_1\sqcup ...\sqcup S_n}(\eta_{S_1}\otimes ...\otimes \eta_{S_n})\circ \triangle_{S_1,...,S_n}(g). $$
The sum is over weak compositions of $I$ with $S_1\cup ...\cup S_i$ convex for all $i$. This counts the number of functions $f: I\rightarrow [n]$ such that $f^{-1}[i] = S_1\cup... \cup S_i$, which are the convex functions. By Proposition \ref{extremal=convex}, this is the same as the number of extremal functions $f: I \rightarrow [n]$. Similar arguments can be applied on $\chi^\zeta_I$.  
\end{proof}

\begin{corollary}
The number of (strictly) extremal functions $f: I \rightarrow [n]$ is a polynomial in $n$.
\end{corollary}

By the fact that $\overline{\zeta} = \eta^{-1}$, and Proposition \ref{general_reciprocity}, we obtain the Edelman-Jamison reciprocity \cite[Theorem 4.7]{EJ}.

\begin{corollary}\label{reciprocity_eta_zeta}
    $(-1)^{|I|}\chi^\eta_I(g)(-n) = \chi^\zeta_I(g)(n)$.
\end{corollary}


Now we want to describe $\varphi$ and $\varphi'$ combinatorially and compute the associated polynomial invariants.  Let $g$ be a convex geometry on the ground set $I$. A subset $S\subseteq I$ is called \emph{totally convex} in $g$ if for all $A\subseteq S$, $g(A) = A$. The set of \emph{extremal points} of $g$ is $\Ex(g) := \{x\in I \mid g(I\backslash x) = I \backslash x\} = \Ex(I)$. Let $\ex(g) = |\Ex(g)|$. It is easy to see that when $g = g_p$ for some partial order $p$, then the number of totally convex sets with respect to $g$ is $2^{\min(p)}$, and $\ex(g) = \max(p)$.

\begin{proposition}
$\varphi_I(g)$ is the number of totally convex sets in $g$. $\varphi_I'(g) = 2^{\ex(g)}$.
\end{proposition}
\begin{proof}
     We have that $$\varphi(g) = \sum_{S\sqcup T = I, S\text{ is convex}}\zeta(g|_S)\eta(g_S).$$
This counts the number of ordered decompositions $I=S\sqcup T$ such that $S$ is convex and $g|_S$ is discrete. For such compositions, for all $A\subseteq S$, $g|_S(A) = g(A)\cap S = g(A) = A$. Hence this is the same as counting totally convex sets in $g$. Similarly, $$\varphi_I'(g) = \sum_{S\sqcup T = I, \text{ $S$ is convex}}\eta(g|_S)\zeta(g/_S).$$ This counts the number of ordered decompositions $I = S\sqcup T$ such that $S$ is a convex set, and $g/_S$ is discrete. This is exactly the number of subsets in $\Ex(I)$, which is $2^{\ex(g)}$.
\end{proof}

For $n \in \mathbb{N}$, recall $\llbracket n\rrbracket = \{\overline{n},...,\overline{1},1,...,n\}$ together with the linear order $\overline{1} < 1 < ... < \overline{n} < n$. A function $f:I\rightarrow \llbracket n \rrbracket$ is called \emph{enriched convex} with respect to $g$ if it satisfies the following properties.
\begin{enumerate}
    \item[$(1)$] For each convex set $A\subseteq \llbracket n \rrbracket$, $f^{-1}(A)$ is convex. 
    \item[$(2)$] Denote $f^{-1}(\llbracket i \rrbracket)$ as $A_i$ for each $1\leq i \leq n$. Then $g_{A_{i-1}:A_i}$ restricted on $f^{-1}(\overline{i})$ is discrete.
\end{enumerate}

\begin{theorem}\label{poly_cg_varphi}
    Let $\chi^\varphi$ be the polynomial invariant associated with $\varphi$. Then  $\chi^\varphi_I(g)(n)$ is the number of enriched convex functions $f: I\rightarrow \llbracket n \rrbracket.$ 
\end{theorem}
\begin{proof}
    We have $$\chi^\varphi_I(g)(n) = \sum_{I = S_1\sqcup...\sqcup S_n}\varphi_{S_1}\otimes...\otimes \varphi_{S_n} \circ \triangle_{S_1,...,S_n}(g)$$ $$ = \sum_{I = S_1\sqcup...\sqcup S_n, S_1\cup...\cup S_i \text{ convex}}\varphi_{S_1}({g_{A_0:A_1}})...\varphi_{S_n}({g_{A_{n-1}:A_n}}). $$

For $1\leq i \leq n$, each choice of $S_1\cup...\cup S_i$ corresponds to a distinct choice of $f^{-1}(\llbracket i \rrbracket)$. By definition of $\varphi$ we have $$\varphi_{S_i}({g_{A_{i-1}:A_i}}) = \sum_{S_i = M\sqcup N, M\text{ convex in }{g_{A_{i-1}:A_i}} }\zeta({g_{A_{i-1}:A_i}|_M}).$$ That is, $\varphi_{S_i}({g_{A_{i-1}:A_i}})$ is the number of convex $M\subseteq S_i$ such that $g_{A_{i-1}:A_i}|_M$ is discrete. Hence each $M$ corresponds to a distinct choice of $f^{-1}(\overline{i})$. Note that $M$ is convex in $g_{A_{i-1}:A_i}$, so we have $g(M\cup A_{i-1})\cap (A_i\backslash A_{i-1}) = M$. Since $g(A_{i-1}) = A_{i-1}\subseteq M\cup A_{i-1}\subseteq A_i = g(A_i)$, we have $M\cup A_{i-1}$, which corresponds to $f^{-1}(\llbracket i-1\rrbracket \cup \{\overline{i}\})$, is convex.
\end{proof}

\begin{corollary}
   The number of enriched convex functions $f:I \rightarrow \llbracket n \rrbracket$ with respect to $g$ is a polynomial in $n$.
\end{corollary}

Note if $g = g_p$ for some partial order $p$, then $\chi^\varphi_I(g)$ is the enriched order polynomial of $p$.

Billera, Hsiao and Provan \cite{BHP} introduced \emph{enriched extremal functions} as a generalization of \emph{enriched $P$-partitions} \cite{Stembridge}. That is, a function $f: I \rightarrow$ $\llbracket n\rrbracket$ is called \emph{enriched extremal} \cite[Definition 4.3]{BHP}  with respect to $g$ if \begin{enumerate}
    \item[$(1)$] For every convex set $A$ there exists $a\in \Ex(A)$ such that $f(a) = f_A$ where $f_A := \min\{f(a) \mid a\in A\}$.
    \item[$(2)$] For $a\in I$, if $f(a) < 0$, then $a\in \Ex(\{b\in I \mid f(b) \geq f(a)\})$.
\end{enumerate}

Let $\chi^{\varphi'}$ be the polynomial invariant associated with $\varphi'$. 
\begin{theorem}\label{poly_cg_varphi'}
$\chi^{\varphi'}_I(g)(n)$ counts the number of enriched extremal functions $f: I \rightarrow \llbracket n\rrbracket$.

\end{theorem}
\begin{proof}
     We recall the notion of \emph{minors} on convex geometries, which comes from the iterated coproduct. For an ordered weak composition $I = S_1\sqcup S_2\sqcup ...\sqcup S_n$ with $S_1\sqcup ... S_i$ convex for each $i$, let $A_i$ denote $S_1 \sqcup ... \sqcup S_i$ and set $A_0 = \varnothing$. We have that $$\triangle_{S_1,..., S_n}(g) = {g_{A_0:A_1}\otimes g_{A_1:A_2}\otimes ...\otimes g_{A_{n-1}:A_n}}. $$

Hence the polynomial invariant associated with $\varphi'$ is as follows: $$\chi^{\varphi'}_I(g)(n) = \sum_{I = S_1\sqcup...\sqcup S_n} \varphi'_{S_1}\otimes ...\otimes \varphi'_{S_n}\circ \triangle_{S_1,...,S_n} (g) $$ $$=\sum_{I = S_1\sqcup... \sqcup S_n, A_i = S_1\sqcup ... \sqcup S_i \text{ is convex}}\varphi'_{S_1}({g_{A_0: A_1}}) ... \varphi'_{S_n}({g_{A_{n-1}:A_n}})$$ $$= \sum_{I = S_1\sqcup... \sqcup S_n, A_i = S_1\sqcup ... \sqcup S_i \text{ is convex}} 2^{\ex(g_{A_0:A_1})}...  2^{\ex(g_{A_{n-1}:A_n})}. $$ This counts the number of functions satisfying certain properties. Before introducing these properties, we define $\overline{[n-i]} \subseteq \pm[n]$ as the set $\{n,\overline{n}, n-1, \overline{n-1},...,n-i, \overline{n-i}\}$, preserving the order on $\llbracket i\rrbracket$. The functions counted by $\chi^{\varphi'}_I(g)(n)$ are those $f: I \rightarrow \llbracket n\rrbracket$ such that \begin{itemize}
    \item[(1)'] $f^{-1}(\overline{[n-i]})$ is convex for all $0\leq i \leq n-1$, and specifically $f^{-1}(\overline{[n-i]}) = S_1 \sqcup ...\sqcup S_{i+1} = A_{i+1}$.
    \item[(2)'] for each $i$, $f^{-1}(-(n-i)) \in \{a\in S_{i+1}:S_{i+1}\backslash\{a\} \text{ is convex in }g_{A_i:A_{i+1}}\}$.
\end{itemize}

We claim that such functions $f$ are equivalent to the enriched extremal functions, i.e., (1)'+(2)' $\Leftrightarrow (1)+(2)$ where $(1), (2)$ are the labels appeared in definition of enriched extremal functions. We first show the following lemma.

\begin{lemma}\label{all convex}
For each $0\leq i\leq n-1$, $f^{-1}(\overline{[n-i+1]}\cup n-i)$ is convex.
\end{lemma}

By (2)', $S_{i+1}\backslash f^{-1}(\overline{n-i})$ is convex in $g_{A_i:A_{i+1}}$. That is, $$g(A_i\cup S_{i+1}\backslash f^{-1}(\overline{n-i}))\cap S_{i+1} = S_{i+1}\backslash f^{-1}(\overline{n-i}). $$ Since $A_i \subseteq g(A_i\cup S_{i+1}\backslash f^{-1}(\overline{n-i})) \subseteq g(A_{i+1}) = A_{i+1}$, we have $g(A_i\cup S_{i+1} \backslash f^{-1}(\overline{n-i})) = A_i \cup S_{i+1}\backslash f^{-1}(\overline{n-i})$. Hence $A_i \cup S_{i+1}\backslash f^{-1}(\overline{n-i}) =f^{-1}( \overline{[n-i+1]}\cup n-i$) is convex. So the lemma holds. In other words, for any $x\in \llbracket n\rrbracket$, $f^{-1}(\{n,\overline{n}, n-1, \overline{n-1},...,x\})$ is convex.

(1)'+(2)' $\Rightarrow (1)$: suppose $1$ is false, that is, for some convex set $A\subseteq I$, $\Ex(A) \cap \{a\in A: f(a) = f_A\} = \varnothing$. Let $x$ be the second smallest elements in ${f(a): a\in A}$, let $S_x$ denote the set $\{n, \overline{n},...,x\}$ where elements are listed in decreasing order from $n$ to $x$. Then $\Ex(A)\subseteq f^{-1}(S_x)$, which is convex by Lemma \ref{all convex}. Then $\overline{\Ex(A)} \subseteq \overline{f^{-1}(S_x)} = f^{-1}(S_x)$. Since $A$ is convex, by Lemma \ref{lemma2eq}, $A = \overline{\Ex(A)}$, hence $A\subseteq f^{-1}(S_x)$, contradiction. 

(1)'+(2)' $\Rightarrow (2)$: let $a\in A$ such that $f(a) = \overline{n-i}$ for some $0\leq i \leq n-1$, i.e., $a\in S_{i+1}$. By (2)', we have $$g_{A_i:A_{i+1}}(S_{i+1}\backslash \{a\}) = g(S_{i+1}\backslash\{a\}\cup A_i)\cap S_{i+1} = S_{i+1}\backslash \{a\}. $$

Observe that $A_i \subseteq g(S_{i+1}\backslash\{a\}\cup A_i) \subseteq A_{i+1}$, then we have $g(S_{i+1}\backslash\{a\}\cup A_i) = g(A_{i+1}\backslash\{a\})= A_{i+1}\backslash\{a\}$. Hence $a\in \Ex(A_{i+1}) = \Ex\{b\in I: f(b) \geq f(a)\}$ as desired.

(1)+(2) $\Rightarrow (1)'$: the proof is essentially the same as that of Proposition \ref{extremal=convex} by replacing $f^A$ to $f_A$. 

(1)+(2) $\Rightarrow (2)'$: fix any $0\leq i \leq n-1$. Let $a\in f^{-1}(\overline{n-i})$. By (2), $a\in \Ex\{b\in I: f(b) \geq f(a)\} =\Ex(A_{i+1})$. Hence $g_{A_i:A_{i+1}}(S_{i+1}\backslash\{a\}) = g(A_{i+1}\backslash \{a\})\cap A_i$ = $(A_{i+1}\backslash \{a\})\backslash A_i  = S_{i+1}\backslash \{a\}$, and we can deduce (2)'.
\end{proof}

\begin{corollary}\cite{BHP}
The number of enriched extremal functions $f: I \rightarrow \llbracket n \rrbracket$ with respect to $g$ is a polynomial in $n$.
\end{corollary}

\begin{example}
Let $g$ be the loopless convex geometry on the set of three colinear points $I = \{x,y,z\}$ with convex sets $\varnothing$, $\{x\}$, $\{y\}$, $\{z\}$, $\{x,y\}$, $\{y,z\}$, $\{x,y,z\}$. We omit the brackets in the computation below. Then when $n = 1$, $$\chi'_I(g)(1) = 2^{\ex(g)} = 2^2 =4. $$ This corresponds to the four enriched extremal functions from $I$ to $\llbracket 1\rrbracket$ which are (1,1,1), (-1,1,1), (1,1,-1), (-1,1,-1).

When $n = 2$, $$\chi'_I(g)(2) = _{x\sqcup yz} 2\cdot 2 + _{y\sqcup xz} 2\cdot2^2 + _{z\sqcup xy}2\cdot 2 + _{xy\sqcup z}2^2\cdot2 + _{yz\sqcup x}2^2\cdot 2+ _{xyz\sqcup \emptyset}  2^2  \cdot 2^0$$ $\hspace{9mm}$
$+ _{\emptyset \sqcup xyz}2^0 \cdot 2^2  =40.$

\begin{itemize}

    \item[] $x\sqcup yz:$ (2,1,1), (2,1,-1), (-2,1,1), (-2,1,-1),
    \item[] $y\sqcup xz$: (1,2,1), (1,2,-1), (-1,2,1), (-1,2,-1), (1,-2,1), (1,-2,-1), (-1,-2,1), (-1,-2,-1),
    \item[] $z\sqcup xy$: (1,1,2), (-1,1,2), (1,1,-2), (-1,1,-2),
    \item[] $xy\sqcup z$: (2,2,1), (2,-2,1), (-2,2,1), (-2,-2,1), (2,2,-1), (2,-2,-1), (-2,2,-1), (-2,-2,-1),
    \item[] $yz \sqcup x$: (1,2,2), (1,-2,2), (1,2,-2), (1,-2,-2), (-1,2,2), (-1,-2,2), (-1,2,-2), (-1,-2,-2),
    \item[] $xyz\sqcup \varnothing$: (1,1,1), (-1,1,1), (1,1,-1), (-1,1,-1),
    \item[] $\varnothing \sqcup xyz$: (2,2,2), (-2,2,2), (2,2,-2), (-2,2,-2).
\end{itemize}

These are all enriched extremal functions $I \rightarrow \llbracket2\rrbracket$.
\end{example}

Since $\varphi$ and $\varphi'$ are odd, by Proposition \ref{self_reciprocity}, we obtain Billera-Hsiao-Provan reciprocity (in $\chi^{\varphi'}$ case) and obtain another reciprocity result for $\chi^\varphi$.

\begin{corollary}
    \begin{itemize}
    \item[]
        \item[(1)] $\chi^{\varphi'}_I(g)(-n) = (-1)^{|{I}|}\chi^{\varphi'}_I(g)(n).$  \textnormal{\cite[pp.16]{BHP} \label{reciprocity_cg_varphi'}}
        \item[(2)] $\chi^{\varphi}_I(g)(-n) = (-1)^{|I|}\chi^{\varphi}_I(g)(n).$
    \end{itemize}
\end{corollary}

\subsection{Enumeration on quasisymmetric invariants}\label{enumeration_quasi_inv}

In this section, we discuss the combinatorial and geometric representations of the quasisymmetric invariants (flag $f$-vectors) associated with $\eta$, $\zeta$, and $\varphi'$. For $\varphi$, we raise a question at the end of the section.

Let $g$ be a convex geometry on ground set $I$. We have  a lattice $L_g$ with elements the convex sets in $g$. This lattice is meet-distributive \cite[Theorem 3.3]{Edelman} and is a sub semi-lattice of the Boolean poset of $I$.
Let $V_g$ be the order complex of the proper part of $L_g$. This is a subcomplex of the braid arrangement of $I$. Specifically, faces of $V_g$ correspond to flags in $L_g$.

\begin{example}
Let $g$ be the convex geometry on the set of three colinear points $\{x,y,z\}$. Figure \ref{delta(g)} illustrates a geometric presentation of $g$, the lattice $L_g$ with elements the convex sets, and the order complex $V_g$ (in orange) as a subcomplex in the braid arrangement on $\{x,y,z\}$.

\begin{figure}[h!]
\begin{center}
\includegraphics[width=0.8\textwidth]{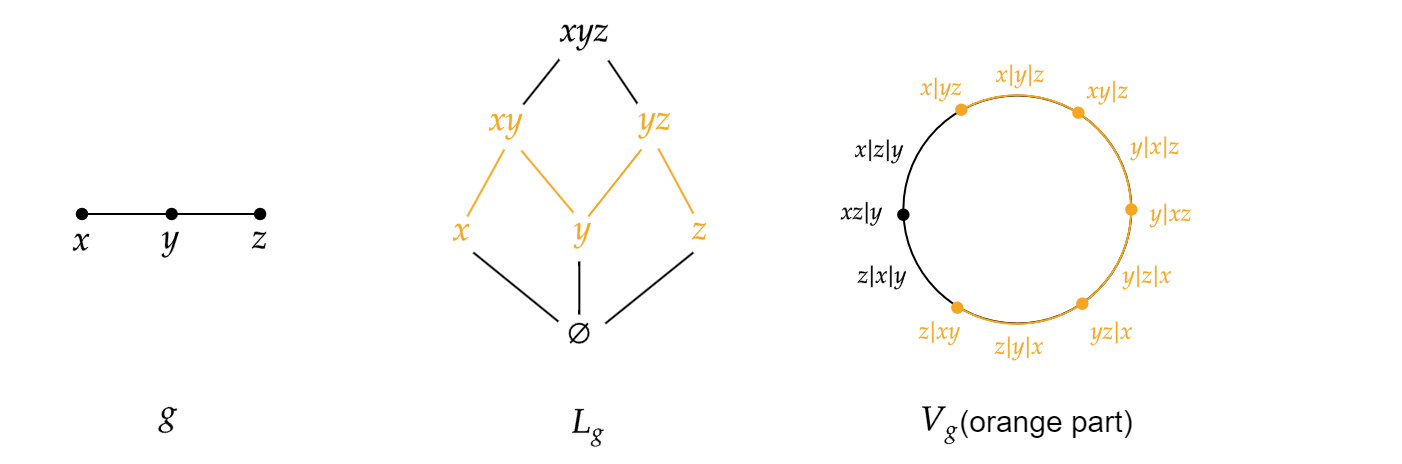}

\caption{$g, L_g, V_g$ of the convex geometry on three colinear points}\label{delta(g)}
\end{center}
\end{figure}
\end{example}

Let $\interior(V_g)$ denote the interior of $V_g$.

\begin{proposition}\label{interior_of_V_g}
Let $g$ be a convex geometry and $F$ be a composition on $I$. Then the following statements are equivalent.
\begin{itemize}
    \item[$(1)$] $F\in \interior(V_g)$.
    \item[$(2)$] $F = (F_1,F_2,...,F_k)$ satisfies $A_i :=F_1\cup...\cup F_i$ is convex for each $0\leq i\leq k$ and $g_{A_i:A_{i+1}}$ is discrete for all $0 \leq i \leq k-1$ with convention $A_0 = \varnothing$.
\end{itemize}
\end{proposition}
\begin{proof}
    Assume $(1)$. That is, $\starr(F)\subseteq V_g$. That is, every extension of $F$ corresponds to a flag of convex sets in $g$. First we have that $F$ itself corresponds to a flag of convex sets in $g$. Also, it follows that for each $0\leq i\leq k-1$, any set $S\subseteq A_i\backslash A_{i-1}$, we have the composition $F' = (F_1,...,F_{i-1}, S, F_i\backslash S, F_{i+1},...,F_k)$ is a flag of convex sets. Hence $g_{A_{i-1}:A_{i}}(S) = g(S\sqcup A_{i-1})\cap(A_i\backslash A_{i-1}) = (S\cup A_{i-1})\cap(A_i\backslash A_{i-1})= S$. So $g_{A_{i-1}:A_i}$ is discrete. We obtain $(2)$.

Now assume $(2)$. Let $G = (G_1,...,G_l)$ be a composition such that $F\leq G$. Then for each $1\leq i \leq l$, we can write $G_1\sqcup...\sqcup G_i = F_1\sqcup...\sqcup F_j \sqcup R$ for some $1\leq j\leq k$ and $R\subseteq F_{j+1}$. Then by the second condition in $(2)$, $g_{A_j:A_{j+1}}(R) = g(R\sqcup A_j)\cap (A_{j+1}\backslash A_j) = R$. Since $A_j\subseteq g(R\cup A_j) \subseteq g(A_{j+1}) = A_{j+1}$, we have $R\sqcup A_j = G_1\sqcup...\sqcup G_i$ is convex. That is, $G$ corresponds to a flag of convex sets. Hence we obtain that $\starr(F)\subseteq V_g$, so $F\in \interior(V_g)$. 
\end{proof}

\begin{remark}
We can extend the above statement to loopless closure operators since we only assume that $g$ is a closure operator in the proof.
\end{remark}

\begin{corollary}

\begin{itemize}
    \item[]
    \item $f^\eta_I(g) = \sum_{F\vDash I}a_\eta(g,F)\BM_F$ with $a_\eta(g,F) = \begin{cases}
1 & \text{if }F\in V_g, \\ 0 & \text{otherwise.}    \end{cases}$
    \item $f^\zeta_I(g) = \sum_{F\vDash I}a_\zeta(g,F)\BM_F $ with $a_\zeta(g,F) = \begin{cases}1 & \text{if $F\in \interior(V_g)$,} \\ 0 & \text{otherwise.} \end{cases}$
\end{itemize}
\end{corollary}

That is, $f_I^{\eta}(g)$ enumerates faces in $V_g$, and $f_I^{\zeta}(g)$ enumerates faces in the interior of $V_g$. 

Let $F = (F_1,...,F_k)$ be a composition of $I$. Let $A_i = \cup_{j=1}^i F_j$ with $A_0 = \emptyset$. Define $\Ex(g,F)$ as the vector $(\Ex(F_i))_i$ where $\Ex(F_i) = \Ex(g_{A_{i-1}:A_i})$ denotes the extreme set under each $g_{A_{i-1}:A_i}$ for all $1\leq i \leq k$. Let $\ex(g,F) = \sum_{i=1}^k|\Ex(F_i)|$.

\begin{proposition}\label{flag_varphi'_BHP_sphere}
$$f^{\varphi'}_I(g) = \sum_{F\vDash I}a_{\varphi'}(g,F)\BM_F $$ with $$a_{\varphi'}(g,F) = \begin{cases}2^{|\ex(g,F)|} & \text{if $F$ corresponds to a flag of convex sets of $g$,} \\ 0 & \text{otherwise.} \end{cases}$$
\end{proposition}

For $g$ a convex geometry on $I$, Billera, Hsiao, and Provan \cite{BHP} constructed an Eulerian poset $Q(g)\cup \{\hat{1}\}$, as the poset of faces of a certain regular CW-sphere $\Sigma(g)$. The elements in $Q(g)$ are signed copies of convex sets with signed values assigned on local extremal points. The order complex $\Delta(\overline{Q(g)})$ (Billera-Hsiao-Provan denoted as $\pm\Delta$), which is a simplicial sphere, is then the barycentric subdivision of $\Sigma(g)$. This colored simplicial sphere can also be obtained by making signed copies of faces with respect to extreme points from the order complex $\Delta(L\backslash \{\hat{0}\})$ constructed in \cite[Theorem 2.1]{BHP}.
\begin{example}\label{BHP_sphere_colinear_points}
    Again, let $g$ be the convex geometry on the set of three colinear points $\{x,y,z\}$. Then we have $g$, $Q(g)$, $\Sigma(g)$ and $\Delta(\overline{Q(g)})$ (the front half) as in Figure \ref{delta(g)}. 
\end{example}

\begin{figure}[h!]
\hspace{30mm}
\includegraphics[width=0.8\textwidth]{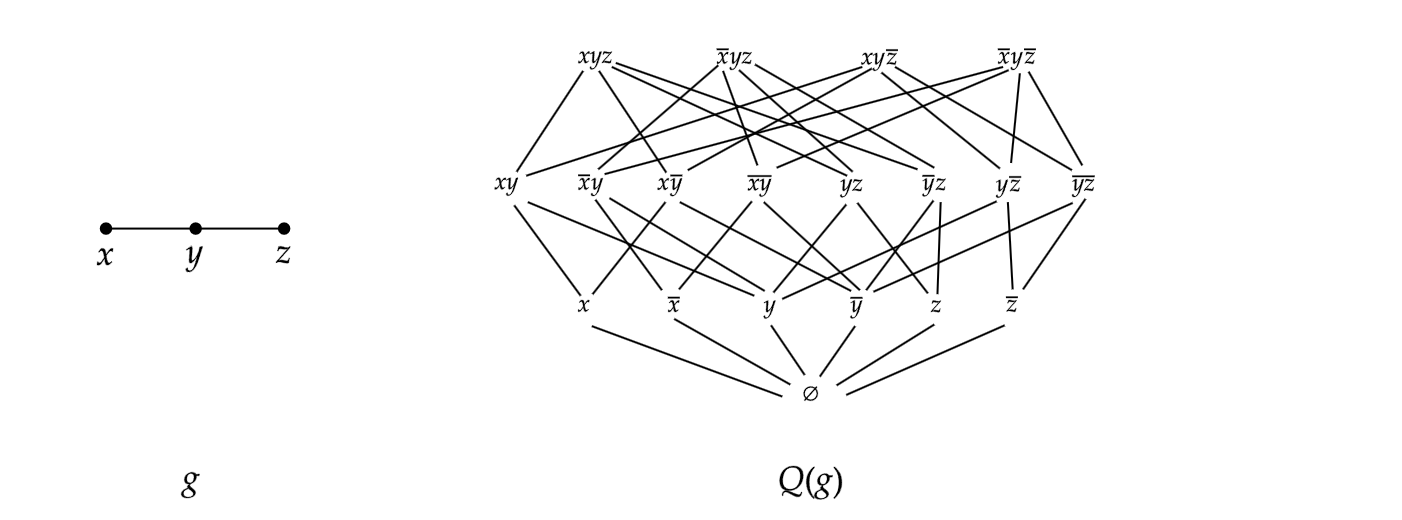}

\end{figure}

\begin{figure}[h!]
\begin{center}
\includegraphics[width=0.8\textwidth]{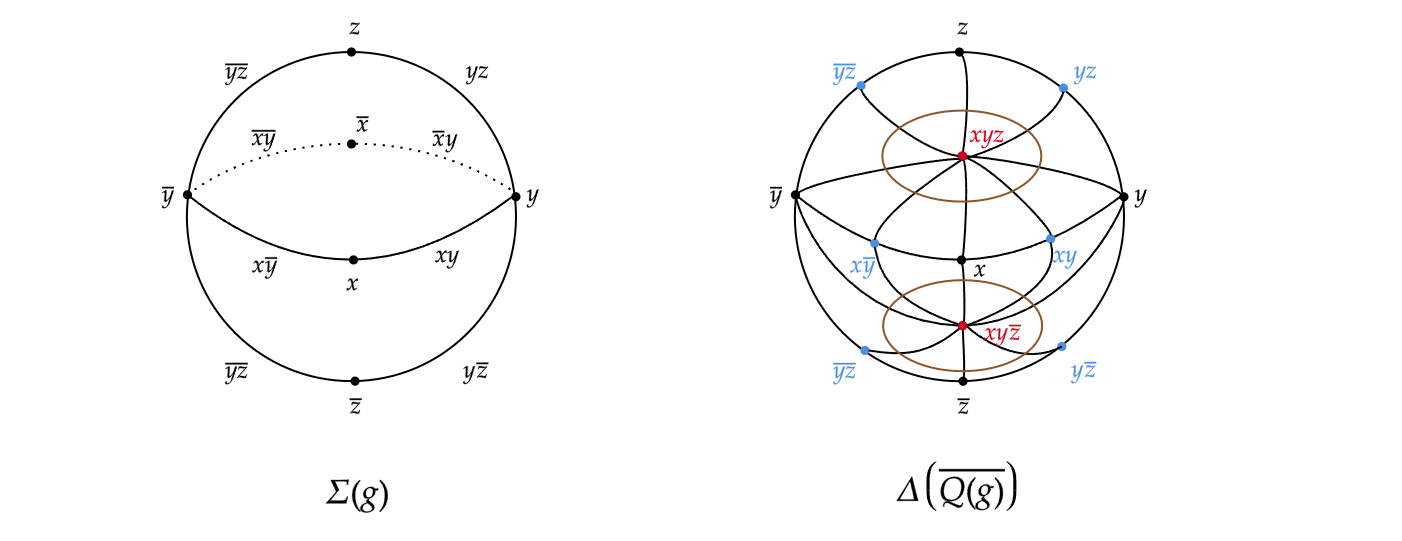}
\caption{$g, Q(g),\Sigma(g),\Delta(\overline{Q(g)})$ of the convex geometry on three colinear points}\label{pic_delta(g)}
\end{center}
\end{figure}

 It is clear from definition that $f^{\varphi'}_I(g)$ enumerates chains in $Q(g)$. Since $Q(g)\cup \{\hat{1}\}$ is Eulerian, $f^{\varphi'}_I(g)$ enumerates chains of intervals of a Eulerian posets. Hence $f^{\varphi'}_I(g)$ is the sum of the flag $f$-vectors of some Eulerian posets.

In $\Delta(\overline{Q(g)})$, we call a face $X$ a \emph{signed copy} of $X_F$ if the vertex set of $X$ is the same as that of $X_F$ up to the sign function. Let $x \in \ex_{g_{A_{i-1}:A_i}}(F_i)$. It is clear that $x\in A_{i}\backslash A_{i-1}$. Unpacking the definitions we have $$x \in \ex_{g_{A_{i-1}:A_i}}(F_i) \Leftrightarrow g_{A_{i-1}:A_i}(F_i\backslash x) = F_i \backslash x \Leftrightarrow g(A_{i}\backslash x) = A_i\backslash x,  x \in A_i\backslash A_{i-1}.$$ Hence we have \begin{equation}\label{enu_varphi'_lemma}
    \ex_{g_{A_{i-1}:A_i}}(F_i) = \ex_g(A_i)\backslash \ex_g(A_{i-1}). 
\end{equation}

By arguments in \cite[proposition 4.1]{BHP} and (\ref{enu_varphi'_lemma}), the number of the signed copies of $X_F$ in $\Delta(\overline{Q(g)})$ is $$\prod_{i=1}^{k}2^{|\ex(A_{i})\backslash \ex(A_{i-1})|} = \prod_{i=1}^k2^{|\ex_{g_{A_{i-1}:A_i}}(F_i)|} = 2^{|\ex(g,F)|}.$$

Hence we have the following geometric description for $f^{\varphi'}$.

\begin{proposition} \label{geometrice_des_f_varphi'}
The quasisymmetric invariant associated with $\varphi'$ 
  enumerates faces in the links of the vertices of $\Delta(\overline{Q(g)})$ corresponding to the facets of $\Sigma(g)$ (colored brown in the figure of $\Delta(\overline{Q(g)})$ in Example \ref{BHP_sphere_colinear_points}). These vertices are exactly the ones whose labels are signed copies of $I$.

\end{proposition}

\begin{example}
    We compute the $(f^{\varphi'}_\alpha(g))_\alpha$ for $g$ the convex geometry on three colinear points $x,y,z$, which is illustrated in Figure \ref{pic_delta(g)}, as follows.
    \begin{enumerate}
        \item $f^{\varphi'}_{(3)}(g) = 4$. This counts the number of links of our interest, that is, the number of signed copies of $\{x,y,z\}$ (we omit the brackets in the figures).
        \item $f^{\varphi'}_{(1,2)}(g) = 16$. This counts the number of $(1,2)$ chains in $Q(g)$, and also the number of vertices corresponding to blue-red edges in $\Delta(\overline{Q(g)})$ in the links of our interest.
        \item $f^{\varphi'}_{(2,1)}(g) = 16$. This counts the number of $(2,1)$ chains in $Q(g)$,and also the number of vertices corresponding to black-red edges in $\Delta(\overline{Q(g)})$ in the links of our interest.
        \item $f^{\varphi'}_{(1,1,1)}(g) = 32$. This counts the number of $(1,1,1)$ chains in $Q(g)$, and also the number of edges (facets) in the links of our interest.
    \end{enumerate}
\end{example}

For $\varphi$, we have a combinatorial description of $f^\varphi$ by unpacking the definitions. \begin{proposition} Let $F = (F_1,...,F_l)\vDash I$, and let $A_0 = \varnothing$, $A_i = F_1\cup...\cup F_i$. Then
    $$f^\varphi_I(g) = \sum_{F\vDash I}a_\varphi(g,F)\BM_F$$ with  $$a_\varphi(g,F) = \begin{cases}\lambda_1\cdot...\cdot\lambda_l & \text{if $F$ corresponds to a flag of convex sets of $g$,} \\ 0 & \text{otherwise.} \end{cases}$$ Here $\lambda_i$ counts the number of totally convex sets in $g_{A_{i-1}:A_i}$.
\end{proposition}

A natural question is whether we can construct an analogue of the discussion for $f^{\varphi'}$ in the case of $f^\varphi$. That is, whether there exists a colored simplicial polytope such that $f^{\varphi}$ enumerates a well-structured collection of faces.

For a more general version, we may ask whether there exists an Eulerian poset such that $f^{\varphi}$ enumerates chains in certain intervals. The case for general convex geometries remains open. However, the case for $g = g_p$, where $p$ is a partial order, is understood. Specifically, we consider the following construction: let $\overline{p}$ denote the reverse partial order of $p$, obtained by reversing all relations in $p$.

\begin{enumerate}
    \item Construct $\Delta(\overline{Q(\overline{p})})$.
    \item For vertex $x\in \Delta(\overline{Q(\overline{p})})$ such that $x$ is not a signed copy of $I$, replace $x$ by $I\backslash x$.
\end{enumerate}
Then the result simplicial sphere, namely $\triangle_\varphi(\overline{Q(p)})$, satisfies $f^\varphi_I$ enumerates (colored) faces in the links of the signed copies of $I$ in $\triangle_\varphi(\overline{Q(p)})$. This follows from the fact that $\Min(p) = \Max(\overline{p})$. 

We have the associated Eulerian poset $Q_\varphi(g_p)$ defined as follows. The elements in the poset are vertices in $\Delta_\varphi(\overline{Q(p)})$, so they are labeled convex sets $S\subseteq I$. The poset relation is defined as follows. $(S,l) \leq (S',l')$ if $S\subseteq S'$ and $S \cup S'$ is a face in $\Delta_\varphi(\overline{Q(p)})$. Indeed, the shape of $Q_\varphi(g_p)$ is the same as that of $Q(g_p)$ (and hence $Q_\varphi(g_p)\cup \{\hat{1}\}$ is Eulerian) and proper elements in $Q_\varphi(g_p)$ are complements to those in $Q(g_p)$ except the signed copies of $I$'s (which are kept). Hence $f^\varphi$ enumerates chains in rank $|I|$ intervals with minimal element $\emptyset$ in a Eulerian poset.

\begin{question}\label{varphi_question}
Extend the above construction of $f^\varphi$ to convex geometries in general.
\end{question}

\section{$ab$ and $cd$-indices and supersolvable convex geometries}\label{ab_cd_section}

In this section, we discuss the coefficients of $ab$-indices associated with $\eta$, $\zeta$ and the coefficients of $cd$-indices associated with $\varphi'$, $\varphi$ on the Hopf monoid of posets $\po$ and convex geometries $\lcg$. On $\lcg$ we focus on the \emph{supersolvable convex geometries}, which admits a geometric presentation when we look at the associated order complex as a subcomplex of the braid arrangement of ground set $I$.

\subsection{$ab$-index associated with partial orders}\label{ab_discussion}

It is known that for pure lexicographically shellable posets, the entries of the flag $h$-vector have a simple combinatorial interpretation as the number of maximal chains with fixed descent set \cite{wachs2006poset}. Indeed, when we discuss $\eta$ on the Hopf monoid of partial orders $\po$, the associated flag $f$-vector $f^\eta(p)$ is the flag $f$-vector of the order complex of (the proper part of) the lattice of order ideals of $p$, which is pure lexicographically shellable. We provide a geometric proof that the coefficients of $ab$-index of $h^\zeta$ and $h^\eta$ come from the number of two-line permutations determined by certain chambers in $V_p$ (maximal chains in $L_p$).

Let $p$ be a partial order on the ground set $I$. recall the corresponding flag $f$-vectors for $\zeta$ are defined as follows. For $\alpha = (\alpha_1,...,\alpha_k) \vDash |I| = n$ and the associated $S(\alpha) = \{\alpha_1, \alpha_1+\alpha_2,...,\alpha_1+...+\alpha_{k-1}\}$ (we can simply write $S$ if there is no confusion with the choice of $\alpha$), $f_\alpha = f_S$ counts the number of $F = (F_1,...,F_k)\vDash I$ such that $F$ is of type $\alpha$ and for $A_i = F_1 \cup...\cup F_i$ with $A_0 = \emptyset$, $g_{A_{i-1}:A_i}$ is discrete.

Recall the flag $h$-vector is determined by the following relation. 

$$h^\zeta_\alpha  = \sum_{\alpha' \leq \alpha} (-1)^{l(\alpha) - l(\alpha')} f^\zeta_{\alpha'}.$$

For $S\subseteq [n-1]$, Let $m(a,b)_S$ denote the degree $n-1$ $ab$-monomial with $b$ on position $s$ for each $s\in S$.

\begin{theorem}\label{ab_zeta_p}
    Let $\ell_0\in V_{\overline{p}}$, i.e., $\ell_0$ is a linear extension of $\overline{p}$, then
$$[m(a,b)_S]\Psi^{\zeta}_p = \# \left \{\ell \in V_p \mid \Descent(\left ( \begin{array}{c}
     \ell_0  \\
     \ell 
\end{array} \right )) = S\right \}.$$
\end{theorem}
That is, the coefficients of the $ab$-index associated with $\zeta$ counts the number of permutations $\left ( \begin{array}{c}
     \ell_0  \\
      \ell
\end{array} \right )$ with descents with respect to $\ell_0$ on positions of $b$'s, such that the base linear order $\ell_0$ satisfies $\ell_0 \in V_{\overline{p}}$ and $\ell$ is a linear extension of $p$.
\begin{proof} Fix an $\ell_0\in V_{\overline{p}}$. For each $\alpha \vDash |I|$, let $\mathrm{S}_\alpha$ denote the set of all $F\vDash I$ with type $\alpha$ such that $F$ contributes to the nonzero components of $f^\zeta_\alpha$. That is, $$\mathrm{S}_\alpha = \{ F\in \interior(V_g) \mid \type(F) = \alpha\}.$$ Consider $$L_\alpha = \{F\ell_0 \mid F\in \mathrm{S}_\alpha\}.$$

Note that $|L_\alpha| = f_\alpha$. By definition, each such $F\ell_0$ is a chamber in $V_p$. Specifically, this chamber is associated with the linear order $\ell$ such that $\ell$ is a linear extension of $F$ and $\ell|_{F_i} = {\ell_0}|_{F_i}$. Since the type of $F$'s are fixed, each $F\ell_0$ is distinct for different $F$. Hence $|L_\alpha| = f_\alpha$. Observe that with respect to $\ell_0$, the descents of $\ell$ can only appear on the positions corresponding to $\alpha_1, \alpha_1+\alpha_2...,\alpha_1+...+\alpha_{k-1}$. 

Note it is clear that the map from $\alpha \vDash n$ to $S\subseteq[n-1]$ is bijective. We claim that $$L_\alpha = \left \{\ell \in \interior(V_p)\mid \Des (\left ( \begin{array}{cc}
     \ell_0  \\
     \ell 
\end{array} \right )) \subseteq S \right \}.$$

To show this claim, it is left to prove the $\supseteq$ direction. Consider $\ell$ satisfying the condition on the right hand side. Let $F$ be the composition of $I$ of type $\alpha$ such that $\ell$ is an linear extension of $F$. Then we have $F\in V_p$, $F\ell_0 = \ell$. Since $\ell_0 \in V_{\overline{p}}$, $\overline{\ell_0} \in V_p$. Hence $F\overline{\ell_0} \in V_p$ as it is convex. By the convexity of $V_p$ again we can deduce that $\starr(F) \in V_p$ and hence $F\in \interior(V_g)$. Hence $\ell \in L_\alpha$ so the claim holds.

Now consider $\alpha'$ such that $\alpha\leq \alpha'$.  Then $$L_{\alpha'} = L_\alpha \cup \left \{\ell \in \interior(V_p)\mid S(\alpha) \subsetneq \Des \left ( \begin{array}{cc}
     \ell_0  \\
     \ell 
\end{array} \right ) \subseteq S(\alpha')\right \}.$$

Then use the principle of inclusion-exclusion we have for fixed $\alpha$, $$\#\left \{\ell \in V_p\mid \Des (\left ( \begin{array}{cc}
     \ell_0  \\
     \ell 
\end{array} \right )) = S(\alpha)\right \} = \sum_{\alpha' \leq \alpha}(-1)^{l(\alpha')-l(\alpha)}\# \left\{\ell\in V_p\mid \Des (\left ( \begin{array}{cc}
     \ell_0  \\
     \ell 
\end{array} \right )) = S(\alpha') \right \}.$$
In particular, $\# \left \{\ell\in V_p\mid \Des (\left ( \begin{array}{cc}
     \ell_0  \\
     \ell 
\end{array} \right )) = S(\alpha)\right \} = h^\zeta_\alpha$. \end{proof}

Fix an $\ell_0 \in V_p$. Let $\alpha = (\alpha_1,...,\alpha_k)\vDash |I|$. By definition, the flag $f$-vector $f^\eta_T$ counts the number of faces $F_\alpha$ (corresponding to $F_\alpha\vDash I$) on $V_p$ with type $\alpha$. Since $p$ is a partial order, $V_p$ is convex, we have for each such face $F_\alpha$, there is a chamber (corresponding to a linear order) $\ell\in V_p$ with $\ell = F\ell_0$, with possible descents on position $\alpha_1,...,\alpha_{k-1}$.

By the same arguments as those in Theorem~\ref{ab_zeta_p}, but considering $\ell_0 \in V_p$, we obtain the result for the $ab$-index of the lattice of order ideals of the partial order $p$.

\begin{theorem}
Fix $\ell_0\in V_{{p}}$, i.e., $\ell_0$ is a linear extension of ${p}$. Let $S\subseteq [n-1]$, then
$$[m(a,b)_S]\Psi^{\eta}_p = \# \left \{\ell \in V_p \mid \Descent(\left ( \begin{array}{c}
     \ell_0  \\
     \ell 
\end{array} \right )) = S\right \}.$$

\end{theorem}

\subsection{Supersolvable convex geometries}\label{supersolvable_convex_section}

Stanley introduced \emph{supersolvable} lattice as a generalization of the subgroup lattice of a solvable group \cite{Stanley1972SupersolvableL}. A lattice $L$ is \emph{supersolvable} if it admits a \emph{chief chain} $\cc$, that is, a maximal chain such that for all chains $m$ in $L$, the smallest sublattice in $L$ containing $\cc$ and $m$ is distributive. 

There are several equivalent conditions for a maximal chain to be a chief chain. One of them is that every element in $\cc$ is \emph{rank-modular}. That is, if $\rho$ is a rank function on $L$,  then for each $m\in \cc$ and each $x\in L$, $\rho(m \wedge x) + \rho(m\vee x) = \rho(m) + \rho(x).$

Bj\"orner and Wachs \cite{Bjrner1983OnLS} provided a description of the flag $h$-vectors of a supersolvable lattice in sense of \emph{EL-labeling}.

Armstrong \cite{Armstrong2007TheSO} introduced the notion of \emph{supersolvable convex geometries} in the context of antimatroids. Specifically, $g$ is a \emph{supersolvable convex geometry} if $L_g$ is a supersolvable lattice. 

In this section we provide a geometric characterization of the order complex of $L_g$ as a subcomplex in the braid arrangement of $I$, and then a geometric description of the associated $ab$-index on $\eta$, $\zeta$.

\begin{lemma}\label{in_meet_dis_modularset_and_dissub_generate_dissub}
Let $g$ be a convex geometry on a finite ground set $I$. Let $M$ be a set of rank modular elements and $D$ a distributive sublattice of $L_g$. Then the sublattice in $L_g$ generated by $M$ and $D$ is distributive.
\end{lemma}
\begin{proof} Let $L'$ be the sublattice generated by $M$ and $D$. By assumption, we have for $x,y\in L'$, $x\wedge y = x\cap y$. Hence to show that $L'$ is distributive, it is enough to show that $x\vee y = x\cup y$. To do so we want to show that all $\vee$'s can be replaced by $\cup$'s when we write $x\vee y$ as meets and joins of elements in $M$ and $D$. We induct on the number of elements from $D$ and $M$ that are involved in $x\vee y$. 

The induction base case is covered as follows. Let $x_1,y_1\in D$, $m_1, n_1\in M$.

\begin{enumerate}
\item $x_1\vee y_1 = x_1\cup y_1$, $m_1\vee n_1 = m_1\cup n_1$.

\item $m_1\vee y_1 = m_1\cup y_1$, $x_1\vee n_1 = x_1\cup n_1$.

\end{enumerate}
Each equality holds as follows.
\begin{enumerate}
\item The first statement follows from the fact that $D$ is distributive, and the second statement follows from the fact $M$ is a set of rank modular elements in $L$. That is, $\rank(m_1\vee n_1) = |m_1|+ |n_1| - |m_1\cap n_1|.$
\item Since $M$ is a set of rank modular elements, and $m_1 \in M$, $\rank(m_1\vee y_1) + \rank(m_1 \wedge y_1) = \rank(m_1) + \rank(y_1)$. Hence $|m_1\vee y_1| + |m_1 \cap y_1| = |m_1| + |y_1|$. Hence $|m_1\vee y_1| = |m_1\cup y_1|$ and so $m_1\vee y_1 = m_1\cup y_1$. The other statement follows from commutativity of $\vee$.

\end{enumerate}

Suppose the number of elements from $D$, $M$ involved is $n > 2$. If there are no $\wedge$'s in $x\vee y$, then we can use the axioms of $\vee$, modularity of elements in $M$ and distributivity of $D$ to replace all $\vee$'s by $\cup$'s.

Suppose there exists at least one $\wedge$ in $x\vee y$. Then we can choose a $\wedge$ such that $$x\vee y = ... \vee (B\wedge C) \vee ... $$ That is, there is no parenthesis in $x\vee y$ in which there exists $\wedge$ outside $B\wedge C$. By commutativity of $\vee$, we can write $x\vee y$ as $A \vee (B \wedge C)$. Note $$A\vee (B\wedge C) = A\vee (B\cap C)  \subseteq (A\vee B)\cap (A\vee C). $$

By induction hypothesis, all $\vee$'s in $A\vee B$ and $A\vee C$ can be replaced by $\cup$, and so $A \vee B = A \cup B$, $A\vee C = A\cup C$. Hence $A\vee (B\wedge C) \subseteq (A\cup B)\cap (A\cup C) = A\cup (B\cap C).$ Note that since $A\vee (B\cap C)$ is the smallest convex set containing $A$ and $B\cap C$, we have $A\cup(B\cap C)\subseteq A\vee (B\wedge C)$. Hence $A\vee (B\wedge C) = A\cup (B \cap C)$ and so we can replace all $\vee$'s in $x\vee y$ by $\cup$. Indeed, $x\vee y = x\cup y$. \end{proof}

\begin{theorem}\label{supersolvable_equivalent_condition}
Let $g$ be a convex geometry. Let $p_1$,..., $p_k$ be partial orders such that $V_{p_1}$,..., $V_{p_k}$ are the maximal convex cones in $V_g$.
Then the following statements are equivalent.\begin{enumerate}
\item[$(1)$] $\bigcap_{i=1}^k V_{p_i}$ contains at least one chamber. That is, there exists a partial order $p_0$ on $I$ with $V_{p_0} = \bigcap_{i=1}^k V_{p_i}$. 
    \item[$(2)$] $g$ is supersolvable.
\end{enumerate}
\end{theorem}

For definitions and equivalent conditions of convexity, see \cite[Section 2.1.4]{MS}. Specifically, Theorem~\ref{supersolvable_equivalent_condition} states that the chief chains in the lattice of a supersolvable convex geometry $L_g$ correspond to chambers in the intersection of maximal convex cones in $V_g$. In the lattice description, the intersection of all rank~$|I|$ distributive sublattices in $L_g$ contains at least one chain of length~$|I|$.

\begin{proof} $``\Rightarrow"$: Let $\ell$ be any linear extension of $p_0$. This linear order corresponds to a maximal chain, namely $\cc_\ell$ in $L_g$. We claim that every element in $\cc_\ell$ is rank-modular.

Fix any $m\in \cc_\ell$. For any $x\in L_g$, then by (1), there exists a maximal partial order $p$ such that $L_{p} \subseteq L_g$ and $L_{p}$ contains $m$ and $x$. Then $m \cup x$ and $m\cap x$ are both in $L_{p}$ and hence in $L_g$. Since $L_g$ is a lattice of convex sets, $m\wedge x = m\cap x$. Since $m\vee x$ is the smallest convex set containing $m$ and $x$ and $m\cup x$ is convex, $m\vee x = m\cup x$. Hence $\rank(m\vee x) + \rank(m \wedge x) = |m\cup x| + |m\cap x| = |m| + |x| = \rank(m) + \rank(x)$. Hence $g$ is supersolvable.

$``\Leftarrow"$: Suppose $g$ is supersolvable. Let $\cc$ be a chief-chain. Suppose there is a partial order $p$ such that $L_p\subseteq L_g$ and $\cc\nsubseteq L_p$. Consider the sub-lattice in $L_g$ generated by elements in $L_p$ and $\cc$, call it $L'$. That is, $L'$ consists of elements obtained from taking the join and meet of elements in $L_p$ and $\cc$  recursively. By Lemma \ref{in_meet_dis_modularset_and_dissub_generate_dissub}, we have $L'$ is distributive. Since $L_p\subsetneq L'$, we have every maximum partial order in $g$ must contain elements in $\cc$. Hence $g$ satisfies (1). \end{proof}

An immediate consequence of this characterization is that the collection of all convex geometries arising from partial orders, as discussed in (\ref{partial_order_to_convex_geometries}), is supersolvable.

\begin{example}\label{supersolvable_convex_geometries_examples} Consider the following convex geometry determined by its convex sets. We omit the set brackets. $$g_1 = \{\emptyset,x,y,z,w,xy,xz,xw,xyz,xyw,xzw\}.$$ 
In Figure \ref{supersolvable_example}, $V_{g_1}$ corresponds is the complex in orange. The intersection of maximal convex subcomplex in $V_{g_1}$ is colored in green. Since the intersection contains chambers, $g_1$ is supersolvable. 

\begin{figure}[h!]
\begin{center}
  \includegraphics[width=0.8\textwidth]{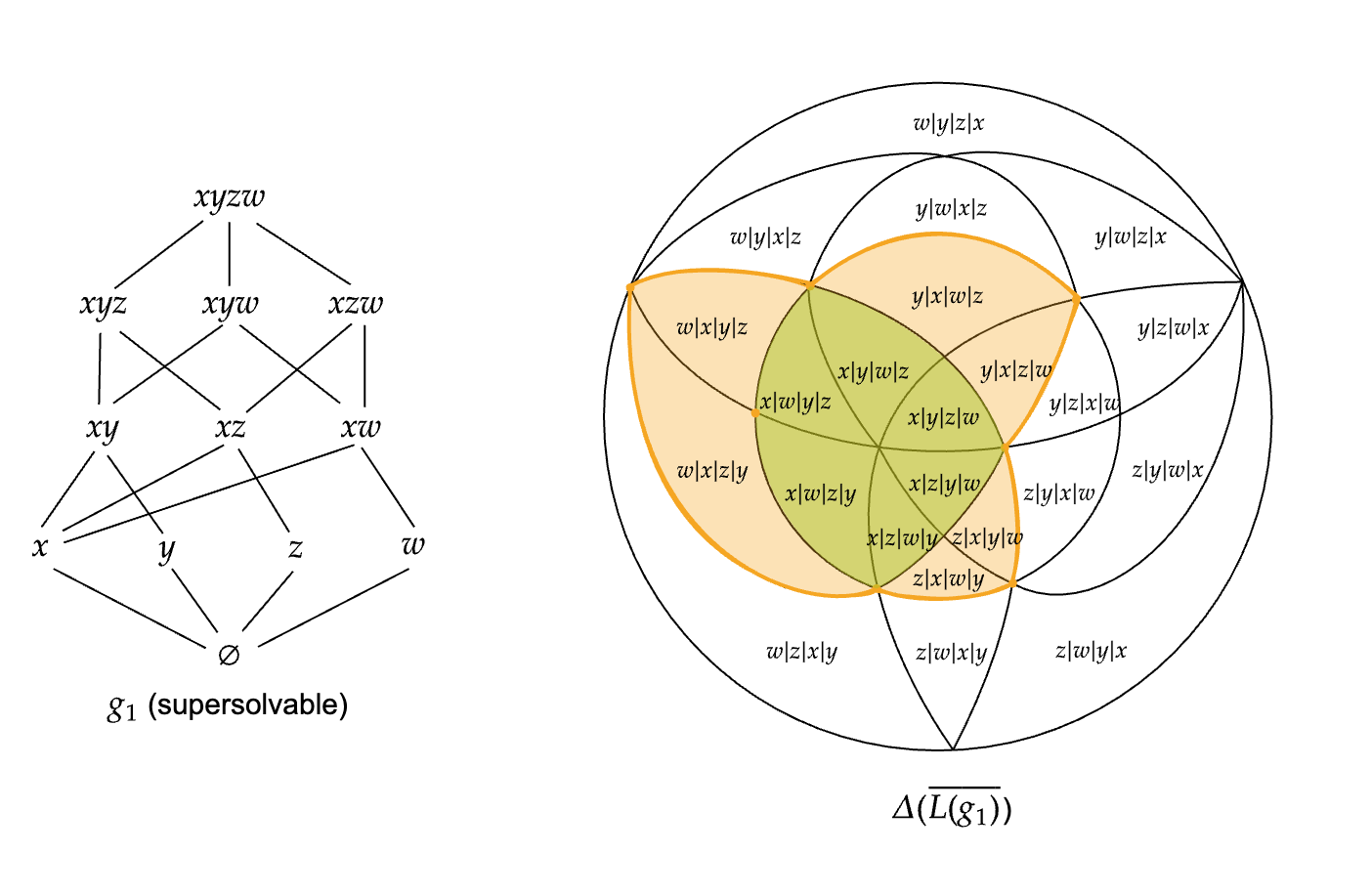}  
\end{center}
\caption{An example of supersolvable convex geometries}\label{supersolvable_example}
\end{figure}
\end{example}

\begin{example}
Consider the following convex geometry determined by its convex sets. We omit the set brackets.
$$g_2 = \{\emptyset,x,y,z,w,xy,yz,zw,xyz,yzw,xyzw\}.$$

In Figure \ref{supersolvable_non_example}, $V_{g_2}$ is the complex in purple which contains the outside chamber $w|z|y|x$. The intersection of maximal convex subcomplex in $V_{g_2}$ is the face colored in blue (a single line segment), which does not contain a chamber. Hence $g_2$ is not supersolvable.

\begin{figure}[h!]
\begin{center}
  \includegraphics[width=0.8\textwidth]{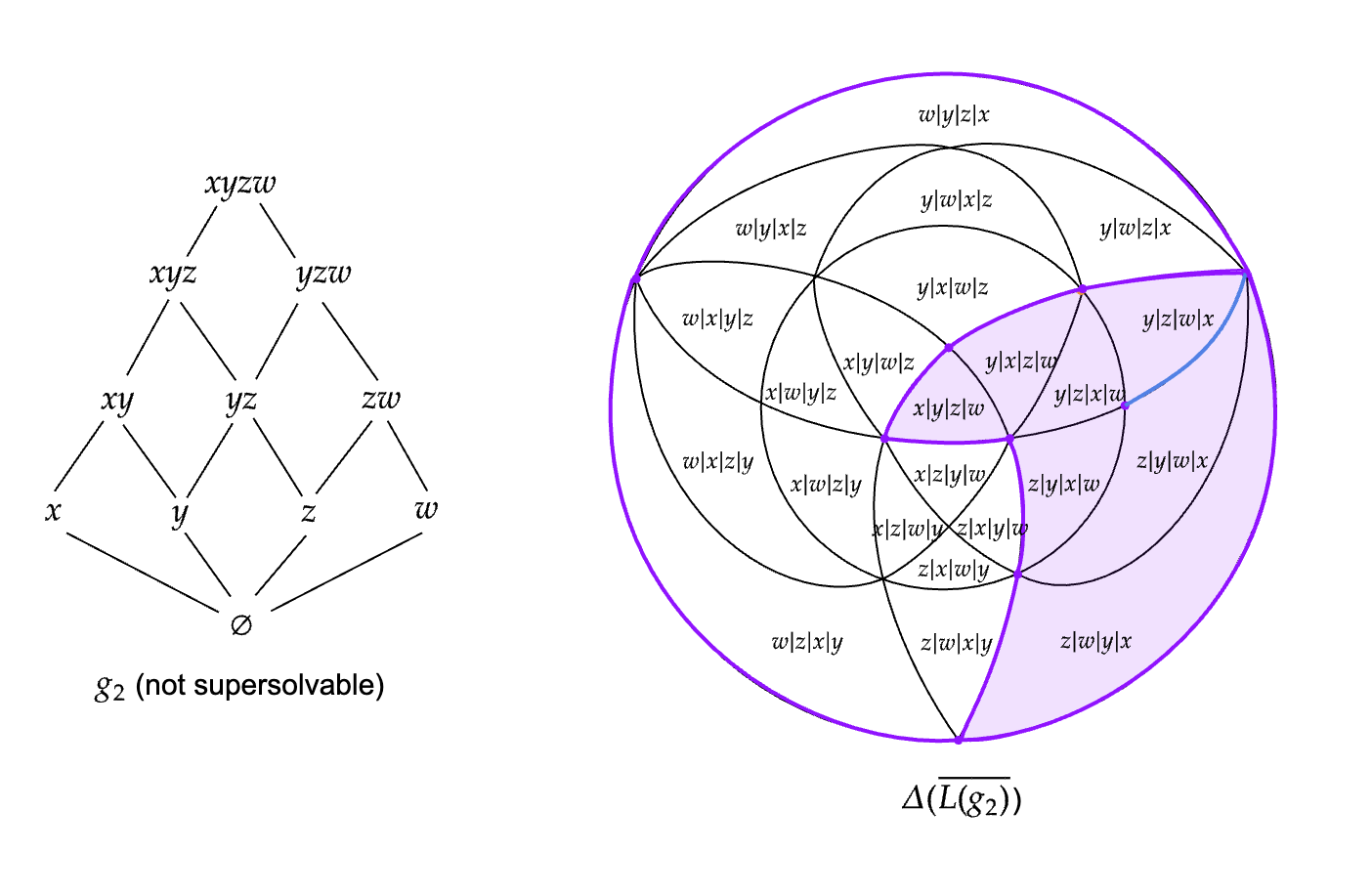}  
\end{center}
\caption{A non-example of supersolvable convex geometries}\label{supersolvable_non_example}
\end{figure}
\end{example}

\newpage
\subsection{$ab$-index associated with supersolvable convex geometries}
Let $g$ be a supersolvable convex geometry, and let $p_0$ be the partial order with $V_{p_0}\subseteq V_p$ for all maximal partial orders $p$ in $g$ according to (1) of Theorem \ref{supersolvable_equivalent_condition}. Let $\overline{p_0}$ be the partial order obtained by reversing relations in $p_0$. Fix any linear order $l_0$, $l_0'$ satisfying $\overline{l_0} \in V_{{p_0}}$, $l'_0 \in V_{p_0}$. We have a geometric description of the $ab$-index associated with $\zeta$, $\eta$ on the Hopf monoid of convex geometries as follows.

\begin{theorem}\label{ab_character_supersolvable}
Let $S\subseteq [n-1]$.
\begin{enumerate}

    \item $[m(a,b)_S]\Psi^{\zeta}_g = \# \left \{\ell \in V_g \mid \Descent(\left ( \begin{array}{c}
     \ell_0  \\
     \ell 
\end{array} \right )) = S\right \}.$
    
\item $[m(a,b)_S]\Psi^{\eta}_g = \# \left \{\ell \in V_g \mid \Descent(\left ( \begin{array}{c}
     \ell'_0  \\
     \ell 
\end{array} \right )) = S\right \}.$

\end{enumerate}
\end{theorem}

\begin{proof}
By Theorem~\ref{supersolvable_equivalent_condition}, $V_{p_0}$ is contained in every maximal convex cone in $V_g$. Therefore, for $\ell'_0 \in V_{p_0}$ and any face $F \in V_g$, the set $\{\ell'_0, F\}$ is contained in some maximal convex cone in $V_g$. It follows that $F \ell'_0 \in V_g$, and thus we can apply the same argument as in the proof of Theorem~\ref{ab_zeta_p}.
\end{proof}

Note that when we look at $\eta$, the corresponding flag $f$ and $h$ vectors $f^\eta_I(g)$ and $h^\eta_I(g)$ are the flag $f$ and $h$ vectors of the order complex of $L_g$.

\begin{example}
   In Example \ref{supersolvable_convex_geometries_examples}, $V_{p_0}$ is in green. we compute the $ab$-index of $g$ associated with $\zeta$, $\eta$ as follows. 
   \begin{itemize}
       \item[(1)]   $\Psi_{\zeta, g}(a,b) = baa + 3aba + 2bba + 2bab + 3abb + b^3$. Fix the base linear order $\ell_0 = w|z|y|x$, so $\overline{\ell_0} \in V_{p_0}$,  then $baa$ corresponds to the chamber $x|w|z|y$, $3aba$ corresponds to the chambers $w|x|y|z$, $y|x|w|z$, $z|x|w|y$, etc.
       \item[(2)]  $\Psi_{\eta,g}(a,b) = a^3 + 3baa + 2aba+ 2aab + 3bab + bba$. Fix the base linear order  $\ell_0' = x|y|z|w$ so $\ell_0' \in V_{p_0}$, then $a^3$ corresponds to the chamber $x|y|z|w$, $3baa$ corresponds to the chambers $w|x|y|z$, $y|x|z|w$, $z|x|y|w$, etc.
   \end{itemize}

\end{example}

\subsection{$cd$-index associated with supersolvable convex geometries}\label{cd_index}

In this section, we discuss the $cd$-indices associated with the canonical odd characters $\varphi$ and $\varphi'$ in the Hopf monoid of convex geometries $\lcg$. We use arguments from combinatorial Hopf algebras and apply Theorem \ref{ab_character_supersolvable} to show that the coefficients of the $cd$-indices associated with $\varphi$ and $\varphi'$ arise from the peaks of certain two-line permutations determined by chambers in $V_g$.

Let $\hcg$ be the $\Bbbk$ vector space with basis the set of all isomorphism classes of loopless finite convex geometries equipped with the following operations. The product of $g_1$, $g_2$ is the direct sum $$g_1 \cdot g_2 = g_1 \oplus g_2.$$ The unit element is the empty geometry, and the coproduct is $$\Delta(g) = \sum_{S\subseteq I, S\text{ convex}}g|_S \otimes g/_{S}.$$

Then $\hcg$ is a graded connected Hopf algebra with the degree of a convex geometry $g$ determined by the size of its ground set $I$. Consider the characters $\eta$, $\zeta$, $\varphi$, and $\varphi'$ on $\hcg$, with the definitions analogue to those in Section \ref{characters_polynomial_invariants}, so that $\hcg$ together with one of these characters is a combinatorial Hopf algebra. By \cite[Theorem 4.1]{MR2196760}, we have the following commutative diagram.

\begin{equation}\label{commutative_diagram}
    \begin{tikzcd}
	\hcg && \qsym \\
	& \Bbbk
	\arrow["\eta"', from=1-1, to=2-2]
	\arrow["{\Upsilon^{\eta}}", from=1-1, to=1-3]
	\arrow["{\eta_{Q}}", from=1-3, to=2-2]
\end{tikzcd}
\end{equation}

In the above diagram, $\qsym$ is the Hopf algebra of quasisymmetric functions. Recall we have a linear basis for $\qsym$
$$   \BM_\alpha := \sum_{i_1<i_2<...<i_k}x^{\alpha_1}_{i_1} x^{\alpha_2}_{i_2}...x^{\alpha_k}_{i_k}.$$

We have a second linear basis of $\qsym$, namely the \emph{fundamental basis}, obtained as follows.
\begin{equation}
    \BF_\alpha := \sum_{\alpha\leq \beta} \BM_\alpha,
\end{equation}
where $\alpha \leq \beta$ means that $\alpha$ can be obtained by merging blocks of $\beta$.

Then, the maps in the commutative diagrams (\ref{commutative_diagram}) are given as follows.

\[\eta_Q(\BM_\alpha) =\eta_Q(\BF_\alpha)  =\begin{cases}
    1 & \text{if } \alpha = (n) \text{ or } (), \\ 0 & \text{otherwise.}
\end{cases}\]

\[ 
\begin{split}
\Upsilon^\eta(g) & = \sum_{\alpha\vDash n}(\eta_\alpha \circ \Delta_\alpha)(g) \cdot \BM_\alpha \\
 & = \sum_{\alpha =(a_1,...,a_k) \vDash n}(\sum_{\footnotesize{{\begin{array}{c}
     {(S_1,...,S_k)\vDash I}  \\
    |S_i| = a_i 
\end{array}}}}(\eta_{S_1}\otimes ...\otimes \eta_{S_k})\Delta_{S_1,...,S_k}(g))\BM_\alpha \\
&  = \sum_{\alpha\vDash n}(\#\{F\in V_g \mid t(F)  = \alpha\})\cdot \BM_\alpha \\
& =  \sum_{\alpha \vDash n}f_\alpha(L_g)\BM_\alpha.
\end{split}
\]

Let $\varphi_Q = \zeta_Q \circ \eta_Q$, with $\zeta_Q = \overline{\eta}_Q^{-1}$. Since $\Upsilon$ is a morphism of combinatorial Hopf algebras, $\Upsilon$ is a group morphism between $\mathbb{X}(\qsym)$ and $\mathbb{X}(\hcg)$ \cite[pp.5]{MR2196760}. Then we have the following diagram commutes. 

\begin{equation}
        \begin{tikzcd}\label{varphi_commute}
	\hcg && \qsym \\
	& \Bbbk & {}
	\arrow["\varphi"', from=1-1, to=2-2]
	\arrow["{\Upsilon^{\zeta}}", from=1-1, to=1-3]
	\arrow["{\varphi_Q}", from=1-3, to=2-2]
\end{tikzcd}
\end{equation}

By \cite[Theorem 4.1, (4.8)]{MR2196760} we also have the following two commutative diagrams.

\begin{equation}
\begin{tikzcd}
	\hcg && \qsym \\
	& \Bbbk
	\arrow["\varphi"', from=1-1, to=2-2]
	\arrow["{\Upsilon^{\varphi}}", from=1-1, to=1-3]
	\arrow["{\eta_Q}", from=1-3, to=2-2]
\end{tikzcd}
\end{equation}
\begin{equation}\label{F'_def}
    \begin{tikzcd}
	\qsym && \qsym \\
	& \Bbbk
	\arrow["{\varphi_Q}"', from=1-1, to=2-2]
	\arrow["\Theta", from=1-1, to=1-3]
	\arrow["{\eta_Q}", from=1-3, to=2-2]
\end{tikzcd}
\end{equation}

Note $\Theta$ is the map introduced by Stembridge \cite[Theorem 3.1]{Stembridge}. This is defined on the fundamental basis $\BF$ as follows. For $J\subseteq [n-1]$, $\Theta (\BF_J) = \BF'_{\Lambda(J)}$, where $$\Lambda(J) = \{i\in J\mid i \geq 2, \hspace{1mm} i-1 \notin J\}.$$ $$\BF'_{\Lambda(J)} = 2^{|\Lambda(J)|+1}\sum_{D\subseteq [n-1], \Lambda(J)\subseteq D\Delta(D+1)}\BF_D$$ 
form a basis of the ``odd subalgebra of $\qsym$". The second equation is \cite[Proposition 3.5]{Stembridge}, where $\Delta$ denotes the symmetric difference. That is, $D\Delta E = (D-E)\cup (E-D)$. Observe that $\{\Lambda(J) \mid J \subseteq [n-1] \}$ is the set of all \emph{peak sets} on $[n]$, which are subsets of $[n-1]\backslash\{1\}$ containing no consecutive elements. Operations of $\Theta$ on $\BM$ basis were derived in \cite[Theorem 2.4]{Hsiao2007STRUCTUREOT} and \cite[4.9]{MR2196760}.

By the universal property of $\eta_Q$ we have that the top triangle of the following 3-dim diagram commutes.

\begin{equation}
\begin{tikzcd}
	\hcg && \qsym \\
	& \qsym \\
	& \Bbbk
	\arrow["{\Upsilon^{\varphi}}", from=1-1, to=2-2]
	\arrow["{\Upsilon^\zeta}", from=1-1, to=1-3]
	\arrow["\Theta"', from=1-3, to=2-2]
	\arrow["{\eta_Q}", from=2-2, to=3-2]
	\arrow["\varphi"', from=1-1, to=3-2]
	\arrow["{\varphi_Q}", from=1-3, to=3-2]
\end{tikzcd}
\end{equation}

We can argue similarly for the character $\varphi'$. By \cite[Theorem 4.1]{MR2196760} we have \begin{equation}
    \begin{tikzcd}
	\hcg && \qsym \\
	& \Bbbk
	\arrow["\zeta"', from=1-1, to=2-2]
	\arrow["{\Upsilon^\zeta}", from=1-1, to=1-3]
	\arrow["{\eta_Q}", from=1-3, to=2-2]
\end{tikzcd}
\end{equation}

Since $\Upsilon$ is a group morphism between $\mathbb{X}(\qsym)$ and $\mathbb{X}(\hcg)$, $\varphi' = \eta * \zeta = \overline{\zeta}^{-1} *\zeta$, $\varphi_Q =\overline{\eta}^{-1}*\eta_Q$, we have the following commutative diagram. 

\begin{equation}
    \begin{tikzcd}
	\hcg && \qsym \\
	& \Bbbk
	\arrow["{\varphi'}"', from=1-1, to=2-2]
	\arrow["{\Upsilon^\zeta}", from=1-1, to=1-3]
	\arrow["{\varphi_Q}", from=1-3, to=2-2]
\end{tikzcd}
\end{equation}

Then by the same argument as in the $\varphi$ case, we can obtain the following 3-dim commutative diagrams, with the top triangle followed from the other commutative triangles and the universality of $\eta_Q$.

\begin{equation} \label{diagram_varphi'}
    \begin{tikzcd}
	\hcg && \qsym \\
	& \qsym \\
	& \Bbbk
	\arrow["{\Upsilon^{\varphi'}}", from=1-1, to=2-2]
	\arrow["{\Upsilon^\zeta}", from=1-1, to=1-3]
	\arrow["\Theta"', from=1-3, to=2-2]
	\arrow["{\eta_Q}", from=2-2, to=3-2]
	\arrow["{\varphi'}"', from=1-1, to=3-2]
	\arrow["{\varphi_Q}", from=1-3, to=3-2]
\end{tikzcd}
\end{equation}

Now we have the two commutative diagrams of our interests.

\begin{equation}\label{quasi_varphi'}
    \begin{tikzcd}
	\hcg && \qsym \\
	& \Bbbk
	\arrow["{\Upsilon^{\varphi'}}", from=1-1, to=2-2]
	\arrow["{\Upsilon^\zeta}", from=1-1, to=1-3]
	\arrow["\Theta"', from=1-3, to=2-2]
\end{tikzcd}
\end{equation}

\begin{equation}\label{quasi_varphi}
    \begin{tikzcd}
	\hcg && \qsym \\
	& \Bbbk
	\arrow["{\Upsilon^{\varphi}}", from=1-1, to=2-2]
	\arrow["{\Upsilon^\eta}", from=1-1, to=1-3]
	\arrow["\Theta"', from=1-3, to=2-2]
\end{tikzcd}
\end{equation}

Let $m(c,d)_S$ denote the degree $n-1$ $cd$-monomial such that the degrees of the initial segments ending in $d$'s are precisely the elements in $S = \{s_1, ..., s_k\} \subseteq [n-1]$. Note by this understanding, we have a bijective correspondence between the set of non-commutative $cd$-monomial of degree $n-1$  (and the corresponding subset of $2^{[n-1]}$ determined by positions of $d$'s) and the set of peak sets of $[n]$, as discussed by Hsiao \cite[section 5.4]{Hsiao2006ASA}: \begin{equation}\label{bijection_cd_peak_set}c^{a_1}dc^{a_2}d...c^{a_k}dc^{a_{k+1}} \leftrightarrow \{\deg(c^{a_1}d), \deg(c^{a_1}dc^{a_2}d), ..., \deg(c^{a_1}d...c^{a_k}d)\}.\end{equation}

Then we have the following description of the $cd$-index associated with $\varphi'.$

\begin{theorem}\label{cd_varphi'_peaks}
Let $\ell_0 \in V_{\overline{p_0}}$, $S\subseteq [n-1]$.
$$[m(c,d)_S]\Phi^{\varphi'}_g = 2^{|S|+1} \cdot \#\left \{\ell \in V_g \mid \Peak(\left ( \begin{array}{c}
     \ell_0  \\
     \ell 
\end{array} \right )) = S\right \}.$$

\end{theorem}

\begin{proof}

Consider the map (\ref{quasi_varphi'}). Recall \[ \begin{split}
   f^\zeta_I(g) & = \sum_{\alpha\vDash n}f^\zeta_\alpha(g)\BM_\alpha = \sum_{\alpha\vDash n} h^\zeta_\alpha(g)\BF_\alpha = \sum_{J\subseteq [n-1]}h^\zeta_J(g)\BF_J. 
\end{split}\]

Since $\varphi'$ is an odd character, by \cite[Proposition 6.5]{MR2196760} applying the map $\Theta$ we have the following equations.
$$f^{\varphi'}_I(g) = \sum_{\alpha\vDash n} f^{\varphi'}_\alpha(g)\BM_\alpha = \sum_{S \text{ peak set}} k^{\varphi'}_S(g)\BF'_S.$$

By definition of $\BF'_S$ and  \cite[proposition 2.2]{MR1982883}, we have  \begin{equation}\label{cd_coefficient_k}
    [m(c,d)_{\Gamma^{-1}(S)}]\Phi^{\varphi'} = {2^{|S|+1}} k^{\varphi'}_S.
\end{equation}

By (\ref{quasi_varphi'}) we have 
\[\begin{split}
    f^{\varphi'}_I(g) = \Theta \Upsilon^\zeta(g)
     = \sum_{J\subseteq[n-1]}h^\zeta_J(g)\BF'_{\Lambda(J)} = \sum_{S \text{ peak set}}(\sum_{J: \Lambda(J) = S}h^\zeta_J(g))\BF'_S.
\end{split}\]

Hence $k^{\varphi'}_S(g) = \sum_{J:\Lambda(J) = S}h^\zeta_J(g)$ for all peak set $S$ and 0 otherwise.

By Theorem \ref{ab_character_supersolvable}, if $g$ is a supersolvable convex geometry, $S\subseteq [|I|]$, then we have $[m(a,b)_S]\Psi_{\zeta,g}(a,b)$ $= h^\zeta_{S}$ is the number of chambers $\ell\in V_g$ such that the two line permutation $\left ( \begin{array}{c}
     \ell_0  \\
      \ell
\end{array} \right )$ has descents on positions of $b$'s, where each $b$ is an element in $S$. The base linear order $\ell_0$ satisfies $\ell_0\in V_{\overline{p_0}}$, i.e., $\overline{\ell_0}$ is a chief chain.

Consider \cite[Equation (3.4)]{MR2052597}, which states that for a peak set $T\subseteq [n-1]$, we have \[ \begin{split}
   \{\sigma\in S_n: \Peak(\sigma) & = T\} = \coprod_{T':\Lambda(T') = T}\{\sigma \in S_n: \Des(\sigma) = T'\}.
\end{split}
\] 

So we have \[ \begin{split}
   V_g \cap \{\sigma\in S_n: \Peak(\sigma) = T\} & = V_g \cap( \coprod_{T:\Lambda(T') = T}\{\sigma \in S_n: \Des(\sigma) = T'\}) \\ & = \coprod_{T':\Lambda(T') = T}\{\sigma \in S_n: \Des(\sigma) = T'\}\cap V_g.
\end{split}
\]

Since $k^{\varphi'}_S(g) = \sum_{J:\Lambda(J) = S}h^\zeta_J(g)$, we have $k^{\varphi'}_S = \# \left\{\ell \in V_g \mid \Peak(\left (\begin{array}{c}
     \ell_0  \\
     \ell 
\end{array} \right ) ) = S\right \}$. By (\ref{cd_coefficient_k}), we obtain the desired equation.\end{proof}

\begin{example}
    In Example \ref{supersolvable_convex_geometries_examples}, we have $\Phi_{\varphi', g}(c,d) = 8c^3 + 8cd + 24dc$. Fix $\ell_0 = w|z|y|x$ so $\overline{\ell_0} = x|y|z|w \in p_0$, then we have 
    \begin{itemize}
        \item[(1)]  $8c^3$ comes from 4 chambers $x|w|z|y$, $x|y|z|w$, $x|z|w|y$, $x|y|w|z$ with no peaks with respect to $\ell_0$. In addition, there is no $d$'s in $8c^3$.
        \item[(2)] $8cd$ comes from 2 chambers $x|w|y|z$, $x|z|y|w$, each of which has one peak on position 3 with respect to $\ell_0$. In addition, there is 1 $d$ in $8cd$.
        \item[(3)] $24dc$ comes from 6 chambers $w|x|y|z$, $w|x|z|y$, $y|x|w|z$, $y|x|z|w$, $z|x|w|y$, $z|x|y|w$, each of which has one peak on position 2 with respect to $\ell_0$. In addition, there is 1 $d$ in $24dc$.
    \end{itemize}
    
\end{example}

By the same arguments (instead we consider $\ell_0 \in V_{p_0}$) applied on (\ref{quasi_varphi}), we have the following description of the coefficients of $cd$-index associated with $\varphi$.

\begin{theorem}\label{cd_varphi_peaks} Let $g$ be a supersolvable convex geometry on ground set $I$. Let $\ell_0 \in V_{p_0}$, $S\subseteq [n-1]$.
$$[m(c,d)_S]\Phi^{\varphi}_{g} = 2^{|S|+1} \cdot \# \left \{\ell \in V_p \mid \Peak(\left ( \begin{array}{c}
     \ell_0  \\
     \ell 
\end{array} \right )) = \Gamma(S)\right \}.$$
\end{theorem}

\section{Supersolvable closure operators}\label{supersolvable_closure_operators}

In this section, we discuss properties of chief chains on supersolvable closure operators, and raise questions about the associated $ab$- and $cd$-indices via the Hopf monoid of closure operators.

\begin{definition}
    Let $\closure$ be a closure operator on ground set $I$. We say $\closure$ is a \emph{supersolvable closure operator} if $L_\closure$, the lattice of closed sets of $\closure$, is supersolvable.
\end{definition}

Let $B(I)$ be the Boolean poset on a finite set $I$. Recall that we denote $V_{B(I)}$ as the Coxeter complex associated with $B(I)$. Let $cl$ be a closure operator on $I$, and let $V_{cl}$ be the subset of faces in $V_{B(I)}$ corresponding to chains in $L_{cl}$. In general, $V_{cl}$ is not a simplicial complex. An example of such a closure operator on $I = \{x, y, z, w\}$ is given by $cl = \{\emptyset, x, y, xyzw\}$, where the operator is represented by its collection of closed sets.

A \emph{preorder}, denoted as $pr$, on $I$ consists of a binary relation which is reflexive and transitive. If the binary relation in $pr$ is anti-symmetric, then it is a partial order. Note $L_{pr}$, the lattice of the lower sets of $pr$, is distributive and is isomorphic to a lattice of order ideals of some partial order. 

Let $c$ and $c'$ be two chains in $B(I)$. Define $pr(c, c')$ to be the smallest preorder on ground set $I$ whose lattice of lower sets contains both $c$ and $c'$. The corresponding lattice of order ideals, denoted $L_{pr(c,c')}$, is obtained by (iterated) intersections and unions of the lower sets in $c$ and $c'$. Denote by $V_{c,c'}$ the cone in $V_{B(I)}$ associated with $pr(c, c')$.

For a chain $c$ in $B(I)$, write $c = c_1 | \dots | c_k$, where each $c_1 \cup \dots \cup c_i$ is an order ideal in $c$. Then $c$ corresponds to a face in $V_{B(I)}$. Recall that $\supp(c)$ denotes the smallest flat in $V_{B(I)}$ containing $c$. The maximal faces in $\supp(c)$ correspond to the chains $c_{i_1} | \dots | c_{i_k}$, where $(i_1, \dots, i_k)$ is a permutation of $(1, \dots, k)$.

\begin{proposition}\label{chief_chain_supersolvable}
    Let $cl$ be a closure operator on $I$ such that $L_{cl}$ is supersolvable. If $\cc$ is a chief chain in $L_{cl}$, then for each maximal chain $m$ in $L_{cl}$, $V_{pr(\cc,m)}\cap \supp(\cc) \subseteq V_{cl}$
\end{proposition}

\begin{proof} 

First we know that $\cc$ is a chief chain iff for each chain $c$ in $L_{cl}$, $L_{\cc\vee c}$, the smallest sublattice in $L_\closure$ containing $\cc$ and $c$, is distributive. Note since each chain $c$ is contained in a maximal chain $m$, if the statement is true for each maximal chain, then it holds for every chain.

Note that, since the meet in $L_{\cc\vee m}$ is intersection, every element in $L_{pr(\cc,m)}$ can be obtained by a union of elements in $L_{\cc \vee m}$. Also, if $x_1,..., x_a$ are elements in $L_{\cc \vee m}$ that cover $\emptyset$, then these are precisely elements in $L_{pr(\cc,m)}$ that are of rank 1.

We write $\cc = \cc_1|...|\cc_k$ as segments in a flag of order ideals in $L_{cl}$ (and in $L_{\cc\vee m}$), i.e., $\emptyset = A_0 \subseteq A_1 \subseteq A_2 \subseteq ...\subseteq A_k = I$ with $A_i = \cc_1 \cup...\cup \cc_i$ is closed for $i > 0$. We show that when we have rank $\leq l$ elements in $L_{cl}$, by applying the union operation we cannot generate any chain corresponding to a sub-flag of $\cc_{i_1}\cup...\cup \cc_{i_k}$ which is not included in $L_{\cc \vee m}.$

We need to show that fix any $A'_{l-1} = \cc_{i_1}\cup...\cup\cc_{i_{l-1}}$ with the corresponding sub-flag of $\cc_{i_1}\cup...\cup \cc_{i_k}$ in $L_{\cc\vee m}$, we cannot obtain any $A'_l = A'_{l-1}\cup \cc'$ ($\cc'\in \{\cc_i \mid 1\leq i \leq k\}$) which is not contained in $L_{\cc\vee m}$. Note the only possibility to generate some $A'_{l-1}\cup \cc'$ which is not included in $L_{\cc\vee m}$ is we have some $\cc'$ with $|\cc'| \geq 2$ and $B_1,...,B_r$ with $\rank(B_i) = l$, $A'_{l-1}\subseteq B_i$ (so $B_i$ covers $A'_{l-1}$), and $\bigcup_{i=1}^r B_i = A'_{l-1}\cup \cc'$ is in $L_{pr(\cc,m)}$. By assumptions, we have $\bigvee_{i=1}^r B_i \supsetneqq \bigcup_{i=1}^r B_i$ in $L_{\cc \vee m}$, and $\rank(\bigvee_{i=1}^r B_i) = l-1+r$.

Note $A'_{l-1}$ is a sub-flag of $\cc_{i_1}|...|\cc_{i_k}$. Let $A'_l$ be the $l$'th element and $A'_{l+1}$ be the $l+1$'th element in the corresponding flag. That is, $A'_l = A'_{l-1}\cup\cc_{i_l}, A'_{l+1} = A'_{l-1}\cup\cc_{i_l} \cup \cc_{i_{l+1}}$. If $A'_l \cap \bigvee_{i=1}^r B_i \neq \emptyset$, then in $[A'_{l-1},I]$, $\bigvee_{i=1}^r{B_i}$ is a rank $r$ element but contains $r+1$ rank 1 elements, which is impossible. If $A'_l \cap \bigvee_{i=1}^r B_i = \emptyset$, again, since the interval from $A'_{l-1}$ and the joint of all rank 1 element in $[A'_{l-1}, I]$ is Boolean, we have $\cup_{i=1}^r B_i \backslash A'_{l-1}$ is no longer $\cc'$, contradiction. Hence $\omega(\cc)$ with $\omega \in S_k$ and $\omega(\cc)\in L_{\cc \vee m}$ is exactly the same as those $\omega(\cc)$ in $L_{pr(\cc \vee m)}$.
\end{proof}

\begin{example}\label{supersolvable_cl}
    Consider the closure operator represented by closed sets $cl = \{\emptyset, x,y,w,xy,xz,xw,$ $yw,xyzw\}$. The corresponding subset of faces in $V_{B(I)}$ is colored in purple and blue in figure \ref{supersolvable_closure}. Checking the definition of a chief chain, we have the maximal chains in $cl$ corresponding to purple segments, i.e., $x|y|zw$, $y|x|zw$, $x|w|yz$, $w|x|yz$  are chief chains in $L_{cl}$. They satisfy the geometric condition in proposition \ref{chief_chain_supersolvable}. The remaining three maximal chains are not chief chains since they fail the condition.

    \begin{figure}[h!]
\begin{center}
  \includegraphics[width=0.8\textwidth]{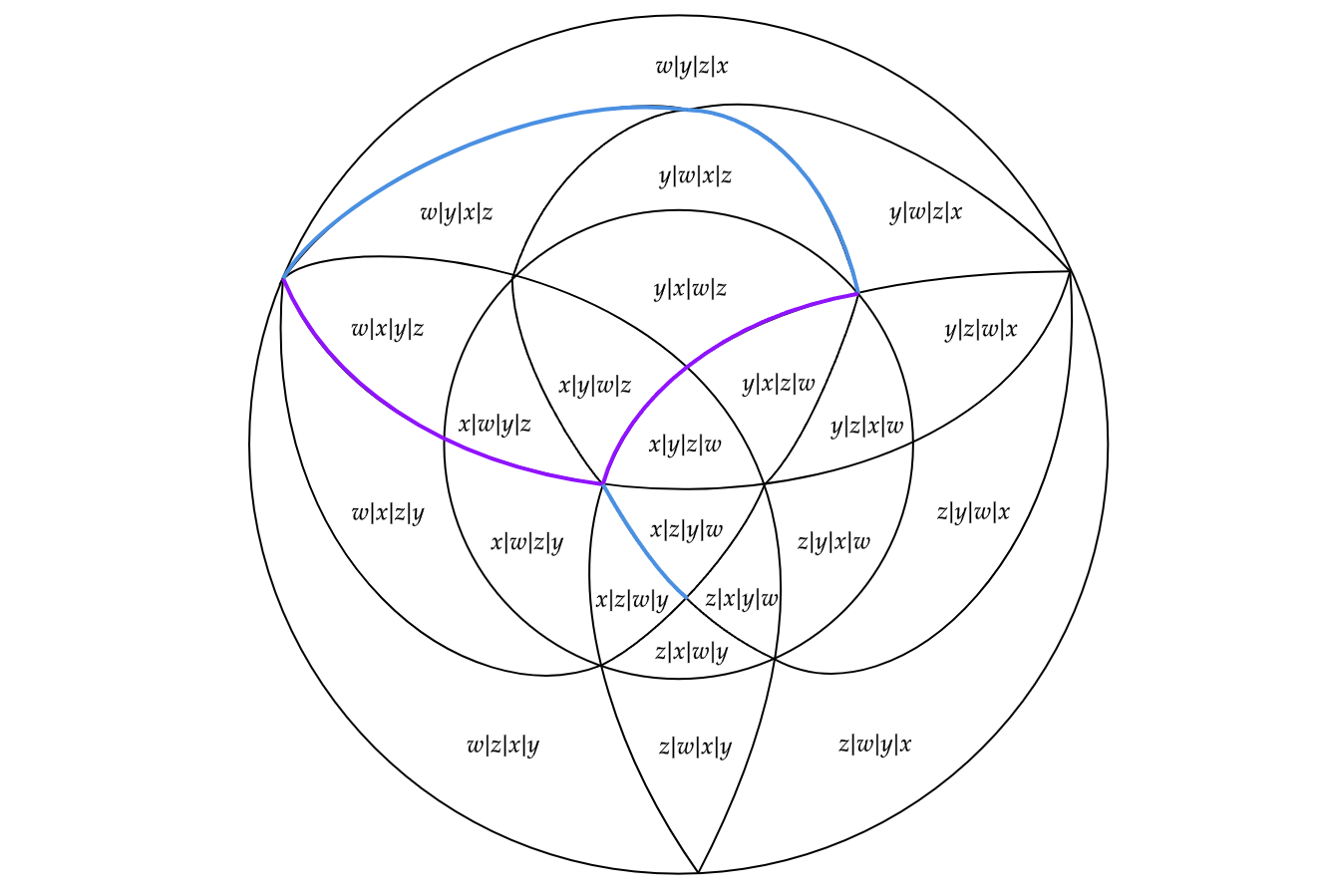}  
\end{center}
\caption{}\label{supersolvable_closure}
\end{figure}
\end{example}

Note the reverse of Proposition \ref{chief_chain_supersolvable}
 is false so it is not an iff statement. For example, consider the closure operator with closed sets $\emptyset, \{a\}, \{b\}, \{c\}, \{a,b\}, \{b,c\},\{a,b,c,d\}$. The lattice of closed sets is supersolvable, but the chamber corresponding to the non-chief chain $\emptyset|a| ab|abcd$ satisfies the geometric condition.

\begin{question}\label{question_classify_chief_chains_geometrically}
    Classify supersolvable closure operators on ground set $I$ on $V_{B(I)}$.
\end{question}

An important classification of supersolvable lattices is that a finite graded lattice of rank $n$ is supersolvable if and only if it is \emph{$S_n$ EL-shellable} \cite[Theorem 1]{McNamara2001ELlabelingsSA}. From this argument we can describe the flag $h$-vector associated with $\eta$ on the Hopf monoid of closure operators, which is also the flag $h$-vector associated with $L_{cl}$, by looking at descents of elements in $S_{\cc} = \{\cc_{i_1}|...|\cc_{i_k}: \{i_1,...,i_k\} = [k]\}$ with respect to the given chief chain $\cc$.

Let $cl$ be a supersolvable closure operator. Let $\cc$ be a chief chain, and write $\cc = (\cc_1,...,\cc_k)$ as segments in a flag of order ideals in $L_{cl}$ (and in $L_{\cc\vee m}$), i.e., $\emptyset = A_0 \subseteq A_1 \subseteq A_2 \subseteq ...\subseteq A_k = I$ with $A_i = \cc_1 \cup...\cup \cc_i$ for $i > 0$. Let $B_{\cc}$ be the Boolean algebra generated by $\cc_i$'s and let $L_{B(\cc)}$ be the corresponding Boolean lattice with rank 1 elements $\cc_i$'s. Note $L_{B(\cc)}$ admits a unique $S_n$ EL-labeling, namely $\mathcal{L}_\cc$, in which the chain $A_0 \subseteq A_1 \subseteq ... \subseteq A_k$ associated with $\cc$ is assigned to $1-2-...-k$.

Fix any chain $m$ in $L_{cl}$, on each step $m_i - m_{i+1}$, we either add some $\cc_i$ or a subset of $c_i$. By rank comparison we cannot add two subsets of $c_i$ in different steps. Hence $L_{\cc\vee m}$ together with the $S_n$ EL-labeling assigning $A_0\subseteq A_1\subseteq ... \subseteq A_k$ to $1-2-...-k$ can be embedding uniquely into $L_{B(\cc)}$ with respect to $\mathcal{L_{\cc}}$.

Hence if we fix the labeling of $\cc$ as $1-2-...-k$, then for each maximal chain $m$ in $L_{cl}$, it is uniquely assigned a labeling permuting ${i_1}-{i_2}-...,-{i_k}$. Hence it corresponds to the chain $\cc_{i_1}-\cc_{i_2}-...-\cc_{i_k}$ in $L_{B(\cc)}$. In $V_{B(I)}$, we may determine the labeling of $V_m$ by comparing the gallery distance (see \cite[section 1.10.3]{MS} for definition) between $V_m$ and $V_\cc$, namely $D(V_m, v_\cc)$ and the gallery distance between $\cc$ and other elements on $\supp(\cc)$. The following proposition follows from that fact that distributive lattices are strongly connected and $L_{\cc\vee m}$ is EL-shellable, admitting an $S_k$ EL-labeling which assigns $\cc$ to $1-2-...-k$ and can be embedded in $\mathcal{L}_{\cc}$.

\begin{proposition}
        Let $\cc' \in L_{B(\cc)}$ with $D(V_\cc,V_\cc') = D(V_\cc, V_m)$ on the same direction, then the $S_n$ labeling of $m$ in $L_{cl}$ is the same as the $S_n$ labeling in $\cc'$ in $L_{B(\cc)}$. 
\end{proposition}

\begin{example}
    In Example \ref{supersolvable_cl}, fix the chief chain $\cc = y|x|zw$, so the associated $S_3$ chain in $\mathcal{L}_\cc$ is $1-2-3$. Then the associated $S_3$ chain for both $y|w|xz$ and $y|wz|x$ is $1-3-2$.
\end{example}

For each $m\in L_{cl}$, let $\cc(m)$ denote the permutation of $\cc$ corresponding to the maximal chain in $L_{B(\cc)}$ such that $D(\cc,\cc') = D(\cc,m)$. Then we can describe the flag $f$-vector $f^\eta_{\alpha}$ as the number of chains $m$ with possible descents of $\left (\begin{array}{c}
     \cc  \\
     \cc(m) 
\end{array}\right)$ appearing on position $\alpha_1, \alpha_1+\alpha_2,...,\alpha_1+...+\alpha_{k-1}$. By similar argument as in Section \ref{ab_discussion}, if we know a chief chain $\cc$ from a supersolvable closure operator $cl$, we can describe the coefficients of flag $h$-vector associated with $\eta$ on the $\lc$ geometrically as follows.

\begin{proposition}
    $[m(a,b)_S]\Psi^\eta_{cl}(a,b)$ is the number of maximal chains $m$ such that the permutation $\left ( \begin{array}{c}
     \cc  \\
      \cc(m)
\end{array} \right )$ admits descents on positions of $b$'s for $b\in S$.
\end{proposition}

These results are weaker than those in the previous sections, as we do not yet have a geometric classification of chief chains for general supersolvable closure operators on $I$. Hence, solving Question~\ref{question_classify_chief_chains_geometrically} is a natural next step. Moreover, since $\zeta$ is the convolution inverse of $\overline{\eta}$, this suggests the conjecture that the $ab$-index associated with $\zeta$ admits a geometric interpretation dual to that of $\eta$. If this is the case, the arguments in Section~\ref{cd_index} can be extended to the context of supersolvable closure operators, providing a geometric description of the $cd$-index associated with $\varphi$ and $\varphi'$.

\section{Acknowledgements}
The author would like to thank Marcelo Aguiar for introducing the topics, providing guidance, and offering helpful comments on this work. The author is also grateful to Margaret Bayer, Yibo Gao, Jeremy Martin, and Yirong Yang for their valuable suggestions that helped improve the presentation of this manuscript.

\printbibliography

\end{document}